\documentclass[11pt]{article}
\usepackage{exscale}
\usepackage{tabularx}
\usepackage[ruled]{algorithm2e}
\usepackage{algorithmic}
\usepackage{amssymb}
\usepackage{latexsym}
\usepackage{amsmath}
\usepackage{amsthm}
\usepackage[dvips]{graphicx,epsfig}
\usepackage{graphicx}
\usepackage{fullpage}        %
\usepackage{float}
\usepackage{verbatim}

\usepackage[bookmarks,pagebackref,
    pdfpagelabels=true, %
    ]{hyperref}

\usepackage{seceqn}
\usepackage{makeidx}
\usepackage{kbordermatrix}
\usepackage{color}

\makeindex

\def\R{\mathbb{R}}
\def\Rk{\mathbb{R}^k}
\def\Rn{\mathbb{R}^n}

\def\eqref#1{{\normalfont(\ref{#1})}}

\def\SDP{\mbox{\boldmath$SDP$}\,}
\def\SDPc{{\mbox{\boldmath$SDP$,}\,}}
\def\EDM{\mbox{\boldmath$EDM$}\,}
\def\edm{\mbox{\boldmath$EDM$}\,}
\def\EDMc{\mbox{\boldmath$EDM$,}\,}

\def\EDM{\mbox{\boldmath$EDM$}\,}

\def\SNL{\mbox{\boldmath$SNL$}\,}

\def\SNLc{\mbox{\boldmath$SNL$,}\,}

\def\EDMC{\mbox{\boldmath$EDMC$}\,}
\def\EDMCc{\mbox{\boldmath$EDMC$},\,}
\def\EDMCp{\mbox{\boldmath$EDMC$}.\,\,}

\def\eqref#1{{\normalfont(\ref{#1})}}

\newtheorem{nmbrs}{Numbering}[section]
\newtheorem{defi}[nmbrs]{Definition}

\newtheorem{prop}[nmbrs]{Proposition}

\newtheorem{lem}[nmbrs]{Lemma}
\newtheorem{thm}[nmbrs]{Theorem}
\newtheorem{cor}[nmbrs]{Corollary}
\newtheorem{rem}[nmbrs]{Remark}

\newtheorem{remark}[nmbrs]{Remark}

\newcommand{\Ss}{{\mathcal S} }
\newcommand{\Ekyb}{{\mathcal E}^{n}(1\!:\!k,\bar D)}
\newcommand{\Eayb}{{\mathcal E}^{n}(\alpha,\bar D)}

\newcommand{\LL}{{\mathcal L} }
\newcommand{\RR}{{\mathcal R} }
\newcommand{\EE}{{\mathcal E} }

\newcommand{\DD}{{\mathcal D} }

\newcommand{\CC}{{\mathcal C} }
\newcommand{\GG}{{\mathcal G} }

\newcommand{\TT}{{\mathcal T} }

\newcommand{\Sh}{{\mathcal S}_H}

\newcommand{\Sc}{{\mathcal S}_C }

\newcommand{\En}{{{\mathcal E}^n} }
\newcommand{\Ek}{{{\mathcal E}^k} }

\newcommand{\Sayb}{{\mathcal S}^{n}(\alpha,\bar Y)}
\newcommand{\Sapyb}{{\mathcal S}_+^{n}(\alpha,\bar Y)}
\newcommand{\Skyb}{{\mathcal S}^{n}(1\!:\!k,\bar Y)}
\newcommand{\Skpyb}{{\mathcal S}_+^{n}(1\!:\!k,\bar Y)}

\newcommand{\rank}{{\rm rank\,}}

\newcommand{\Sn}{{\mathcal S^n}}
\newcommand{\Snp}{{\mathcal S^n_+\,}}

\newcommand{\Sk}{{\mathcal S^{k}}\,}
\newcommand{\Skp}{{\mathcal S^{k}_+}\,}
\newcommand{\Skpp}{{\mathcal S^{k}_{++}}\,}
\newcommand{\Stp}{{\mathcal S^{t}_+}\,}
\newcommand{\Stpp}{{\mathcal S^{t}_{++}}\,}

\newcommand{\MM}{{\mathcal M\,}}
\newcommand{\Mn}{{\mathcal M^n\,}}

\newcommand{\NN}{{\mathcal N}}

\newcommand{\A}{{\mathcal A}}

\newcommand{\N}{{\mathcal N\,}}

\newcommand{\bbm}{\begin{bmatrix}}
\newcommand{\ebm}{\end{bmatrix}}
\newcommand{\bem}{\begin{pmatrix}}
\newcommand{\eem}{\end{pmatrix}}
\newcommand{\beq}{\begin{equation}}
\newcommand{\beqs}{\begin{equation*}}
\newcommand{\bet}{\begin{table}}
\newcommand{\eeq}{\end{equation}}
\newcommand{\eeqs}{\end{equation*}}
\newcommand{\beqr}{\begin{eqnarray}}
\DeclareMathOperator{\KK}{{\mathcal K} }

\DeclareMathOperator{\trace}{{trace}}

\DeclareMathOperator{\diag}{{diag}}
\DeclareMathOperator{\Diag}{{Diag}}

\DeclareMathOperator{\offDiag}{{offDiag}}

\DeclareMathOperator{\relint}{{relint}}

\newcommand{\nc}{\newcommand}
\nc{\arrow}{{\rm arrow\,}}
\nc{\Arrow}{{\rm Arrow\,}}
\nc{\BoDiag}{{\rm B^0Diag\,}}
\nc{\bodiag}{{\rm b^0diag\,}}

\nc{\Mmn}{{\mathcal M}_{m,n} }
\nc{\kwqqp}{Q{$^2$}P\,}
\nc{\kwqqps}{Q{$^2$}Ps}

\nc{\notinaho}{(X,S)\in \overline{AHO}(\A)}
\nc{\inaho}{(X,S)\in AHO(\A)}

\newcommand{\bea}{\begin{eqnarray}}%
\newcommand{\eea}{\end{eqnarray}}%
\newcommand{\beas}{\begin{eqnarray*}}%
\newcommand{\eeas}{\end{eqnarray*}}%
\newcommand{\cone}{{\rm cone\, }}%
\newcommand{\sspan}{{\rm span}\,}
\newcommand{\face}{{\rm face\,}}%
{}

\begin{document}

\bibliographystyle{plain}
\title{
Explicit 
Sensor Network Localization\\
 using\\
 Semidefinite Representations and Facial Reductions
   \footnote{Research supported by Natural Sciences Engineering Research
     Council Canada and a grant from AFOSR.}
}
             \author{
Nathan Krislock
\and
Henry Wolkowicz
}
\date{\today}
          \maketitle
\begin{center}
          University of Waterloo\\
          Department of Combinatorics and Optimization\\
          Waterloo, Ontario N2L 3G1, Canada\\
          Research Report CORR 2009-04
\end{center}


\noindent {\bf AMS Subject Classification:}

\begin{abstract}
\label{pageabstract}
The sensor network localization, \SNL, problem in embedding dimension
$r$, consists of locating the positions of wireless sensors, given only the 
distances between sensors that are within radio range and the positions of a 
subset of the sensors (called anchors).
Current solution techniques relax this problem to a weighted, nearest, 
(positive) semidefinite programming, \SDPc completion problem, by 
using the linear mapping between Euclidean distance matrices, \EDMc and semidefinite matrices.
The resulting \SDP is solved using primal-dual interior point solvers, 
yielding an expensive and inexact solution.

This relaxation is highly degenerate in the
sense that the feasible set is restricted to a low dimensional face of
the \SDP cone, implying that the Slater constraint qualification fails.
Cliques in the graph of the \SNL problem give rise to this degeneracy
in the \SDP relaxation.
In this paper, we take advantage of the absence of the Slater
constraint qualification and derive a technique for the \SNL problem,
with exact data, that explicitly solves the corresponding 
rank restricted \SDP problem. No \SDP solvers are used.
For randomly generated instances,
we are able to efficiently solve many huge instances of this NP-hard problem
to high accuracy,
by finding a representation of the minimal face of the \SDP
cone that contains the \SDP matrix representation of the \EDM.
The main work of our algorithm
consists in repeatedly finding the intersection of subspaces that
represent the faces of the \SDP cone that correspond
to cliques of the \SNL problem.

\end{abstract}
\newpage

\tableofcontents  \label{lists}
\listoftables
\listoffigures
\listofalgorithms
\newpage

\section{Introduction}
\label{sect:intro}

The sensor network localization problem, \SNLc%
\index{\SNL, sensor network localization}%
\index{sensor network localization, \SNL}
consists in locating the positions
of $n$ wireless sensors, $p_i \in \R^r$, $i=1,\ldots,n$, given only the 
(squared) Euclidean distances $D_{ij}=\|p_i-p_j\|_2^2$ 
between sensors that are within a given
radio range, $R>0$, and given the positions of a 
\index{radio range, $R$}
subset of the sensors, $p_i$, $i=n-m+1,\ldots,n$
(called anchors); $r$ is the {\em embedding dimension} of the problem.
\index{embedding dimension ({\em fixed}), $r$}
Currently, many solution techniques for this problem use a relaxation to
a nearest, weighted, semidefinite approximation problem 
\beq
\label{eq:sdprelax}
\min_{Y \succeq 0, \, Y \in \Omega} \left\| W\circ\left(\KK(Y) - D\right) \right\|,
\eeq
where $Y\succeq 0$ denotes positive semidefiniteness, 
$Y \in \Omega$ denotes additional linear constraints, $\KK$ is a
specific linear mapping, and $\circ$ denotes the 
{\em Hadamard (elementwise) product}.
\index{Hadamard product}
This approach requires
semidefinite programming, \SDPc primal-dual interior point  (p-d i-p)
techniques; see, for example,
\cite{AlfakihAnjosKPW:08,AlKaWo:97,biswasphd07,BiswasYe:04,MR2191577,dattorro:05,pongtseng:09}.
This yields an expensive and inexact solution.

The \SNL problem is a special case of the Euclidean Distance Matrix,
\index{\EDM} \EDMc completion problem, \EDMCp
If $D$ is a {\em partial} \EDMc then the completion problem consists in finding
the missing elements (squared distances) of $D$.  
It is shown in \cite{DiKrQiWo:06},
that there are advantages for handling the \SNL problem as an
\EDMCc and ignoring the distinction between the
anchors and the other sensors until after the \EDMC is solved.
In this paper we use this framework and derive an algorithm that 
locates the sensors by exploiting the structure and
implicit degeneracy in the \SNL problem. In particular, we solve the
\SDP problems {\em explicitly} (exactly)
without using any p-d i-p techniques.
We do so by repeatedly viewing \SNL in three equivalent forms:
as a graph realization problem, as a \EDMC, and as a rank restricted
\SDP.

A common approach to solving the \EDMC problem is to 
relax the rank constraint and solve
a weighted, nearest, positive semidefinite
completion problem (like problem~\eqref{eq:sdprelax}) using semidefinite programming, \SDP. 
The resulting \SDP 
is, implicitly, highly degenerate in the
sense that the feasible semidefinite matrices have low rank. In particular,
\label{rankpage}
cliques in the graph of the \SNL problem reduce the ranks of these feasible
semidefinite matrices.
This means that the Slater constraint qualification (strict feasibility)
implicitly fails for the \SDP.  Our algorithm is based on 
exploiting this degeneracy.
We characterize the face of the \SDP cone that corresponds to a given
clique in the graph, thus reducing the size of the \SDP problem.
Then, we characterize the intersection
of two faces that correspond to overlapping cliques.
This allows us to explicitly {\em grow/increase} the size of the cliques by
repeatedly finding the intersection of subspaces that represent
the faces of the \SDP cone that correspond to these cliques.
Equivalently, this corresponds to completing overlapping blocks of the \EDM.
In this way, we further reduce the dimension of the faces until
we get a completion of the entire \EDM. The intersection of the
subspaces can be found using a singular value decomposition (SVD) or by
exploiting the special structure of the subspaces.
No \SDP solver is used.
\index{embedding dimension ({\em fixed}), $r$}
Thus we solve the \SDP problem in a finite number of steps, where the work
of each step is to find the intersection of two subspaces (or, equivalently, 
each step is to find the intersection of two faces of the \SDP cone). 

Though our results hold for general embedding dimension $r$,
\index{embedding dimension ({\em fixed}), $r$}
our preliminary numerical tests involve sensors with 
embedding dimension $r=2$ and $r=3$.  The sensors are in the region $[0,1]^r$. 
There are $n$ sensors, $m$ of which are anchors. The radio range is
$R$ units.
\index{radio range, $R$}

\subsection{Related Work/Applications}
The number of applications for distance geometry problems is large and
increasing in number and importance. The particular case of \SNL has
applications to environmental monitoring of geographical regions, as
well as tracking of animals and machinery; see, for example,
\cite{biswasphd07,dattorro:05}.
There have been many algorithms published recently that solve the \SNL
problem. Many of these involve \SDP relaxations and use \SDP solvers;
see, for example, 
\cite{biswasphd07,biswasliangtohwangye,biswasliangtohye:05,MR2398864,BiswasYe:04,
MR2191577,DiKrQiWo:06}
and more recently  \cite{KimKojimaWaki:09,WangZhengBoydYe:06}.
Heuristics are presented in, for example, \cite{cassioli:08}.
\SNL is closely related to the \EDMC problem; see, for example,
\cite{AlKaWo:97,dattorro:05} and the survey \cite{AlfakihAnjosKPW:08}.

Jin et al \cite{MR2274505,Jin:05} propose the {SpaseLoc} heuristic.
It is limited to $r=2$ and uses an \SDP solver for small localized
subproblems. They then {\em sew} these subproblems together.
So \& Ye \cite{MR2295148} show that the problem of solving a noiseless
\SNL with a unique solution can be phrased as an \SDP and thus can be solved
in polynomial time. They also give an efficient criterion for checking
whether a given instance has a unique solution for $r=2$.

Two contributions of this paper are: we do not use
iterative p-d i-p techniques to solve the \SDP, but rather, we solve it
with a finite number of explicit solutions; 
we start with local cliques and expand the cliques.
Our algorithm has four different basic steps. The first basic step 
takes two cliques for which the intersection contains at least $r+1$ nodes 
and implicitly completes the 
corresponding \EDM to form the union of the cliques.
The second step does this when one of the cliques is a single element.
Therefore, this provides an extension of the algorithm in 
\cite{egwymab04}, where Eren et al
have shown that the family of {\em trilateration graphs}
\index{trilateration graph} admit a polynomial time algorithm for computing
a realization in a required dimension.\footnote{A graph is a trilateration graph 
in dimension $r$ if there exists an 
ordering of the nodes $1, \ldots, r + 1, r + 2, \ldots, n$ 
such that: the first $r + 1$ nodes form a clique,
and each node $j > r + 1$ has at least $r + 1$ edges to nodes earlier
in the sequence.}  Our first basic step also provides an explicit form for finding a
realization of a \index{uniquely localizable graph}
{\em uniquely localizable graph}\footnote{A graph is uniquely localizable in
dimension $r$ if it has a unique realization in $\R^r$ and it does not
have any realization whose affine span is $\R^h$, where $h>r$; see
\cite{MR2295148}.}.
Our algorithm repeatedly finds explicit solutions of an \SDP.
Other examples of finding explicit solutions of an \SDP are given in
\cite{MR97b:90077,Wolk:93}.

The \SNL problem with given embedding dimension $r$ is NP-hard
\index{embedding dimension, $r$}
\cite{hendrickson:91,MR92m:05182,sax79}.
However, from our numerical tests it appears that random problems that
have a unique solution can be solved very efficiently.
This phenomenon fits into the results in
\cite{AmVav:09,FeiKrau:00}.

  \subsection{Outline}

We continue in Section \ref{sect:prels} to present notation and
results that will be used. The facial reduction process is based on the results in Section~\ref{sect:cliqred}. 
The single clique facial reduction is given in Theorem~\ref{thm:onecliquered}; 
the reduction of two overlapping cliques in the
rigid and nonrigid cases is
presented in Theorem~\ref{thm:twocliquered} and
Theorem~\ref{thm:degcompl}, respectively; 
absorbing nodes into cliques in the rigid and nonrigid
cases is given in Corollaries~\ref{cor:absorbsingle} and
\ref{cor:degabsorbsingle}, respectively.
These results are then used in our algorithm in
Section~\ref{sect:algor}. The numerical tests  appear in 
Section~\ref{sect:numerics} and Section~\ref{sect:noisy}.
Our concluding remarks are given in Section~\ref{sect:concl}.

  \subsection{Preliminaries}
\label{sect:prels}

We work in the vector space of 
{\em real symmetric $k \times k$ matrices}, $\Ss^k$,
\index{symmetric $k \times k$ matrices, $\Sk$}
equipped with the 
{\em trace inner product}, $\langle A,B\rangle = \trace AB$.
\index{trace inner product, $\langle A,B\rangle = \trace AB$}
We let $\Skp$ and $\Skpp$
denote the cone of positive semidefinite and positive definite
\index{cone of positive semidefinite matrices, $\Skp$}
matrices, respectively;
\index{cone of positive definite matrices, $\Skpp$}
$A \succeq B$ and $A \succ B$ denote the 
L\"owner partial order, $A-B \in \Skp$ and
\index{L\"owner partial order, $A \succeq B$}
$A-B \in \Skpp$, respectively;
$e$ denote the vector of ones of appropriate dimension;
\index{vector of ones, $e$}
$\RR (\LL)$ and $\NN (\LL)$ denote the range space and null space of the linear transformation $\LL$, respectively;
$\cone(S)$ denote the convex cone generated by the set $S$.
We use the {\sc Matlab} notation $1\!:\!n = \{1,\ldots,n\}$.
\index{$1:n = \{1,\ldots,n\}$}
\index{range space of $\LL$, $\RR (\LL)$}
\index{null space of $\LL$, $\NN (\LL)$}
\index{cone generated by $C$, $\cone (C)$}

A subset $F\subseteq K$ is a {\em face of the cone $K$}, denoted
$F \unlhd K$,  if  \index{face, $F \unlhd K$}
\[
\left( x, y \in K, \ \frac 12(x+y) \in  F\right)
                 \implies \left(\cone \{x,y\} \subseteq F\right).
\]
If $F \unlhd K$, but is not equal to $K$, we write $F \lhd K$.
If $\{0\} \neq F \lhd K$, then $F$ is a {\em proper face} of $K$.
\index{proper face} 
For $S \subseteq K$, we let $\face\!(S)$ denote the smallest face of $K$
that contains $S$. 
A face $F \unlhd K$ is an {\em exposed face} if it is the 
intersection of $K$ with a hyperplane.
\index{exposed face}
The cone $K$ is {\em facially exposed} if every face $F\unlhd K$ is exposed.
\index{facially exposed cone}

The cone $\Snp$ is facially exposed. Moreover, each face $F\unlhd \Snp$ is
determined by the range of any matrix $S$ in the 
relative interior of the face, $S \in
\index{relative interior, $\relint \cdot$}
\relint F$: if $S=U\Gamma U^T$ is the compact spectral decomposition of $S$
with the diagonal matrix of eigenvalues $\Gamma \in \Ss_{++}^t$, then
\label{pagecite}(e.g., \cite{PatakiSVW:99}) 
\beq
\label{eq:facerepr}
F=U \Ss_+^t U^T.
\eeq

A matrix $D=(D_{ij})\in \Sn$ 
with nonnegative elements and zero diagonal is called
a  {\em pre-distance matrix} \index{pre-distance matrix, $D$}. 
In addition, if there exist points $p_1,\ldots,p_n \in \R^r$ such that
\beq  \label{eq:dn}
  D_{ij}=\|{p_i-p_j}\|_2^2, \quad i,j=1,\ldots,n,
\eeq
then $D$ is called a {\em Euclidean distance matrix}, denoted
\index{Euclidean distance matrix, \EDM}
\EDM\@. Note that we work with {\em squared distances}. 
The smallest value of $r$ such that \eqref{eq:dn} holds 
 is called the {\em embedding dimension} of $D$. 
Throughout the paper, we assume that $r$ is given and {\em fixed}.
\index{embedding dimension ({\em fixed}), $r$}
The set of \EDM matrices forms a closed convex cone in $\Sn$, denoted
$\En$.
\index{cone of \EDM, $\Ek$}
If we are given an $n \times n$ partial \EDM $D_p$,
let $\GG=(N,E,\omega)$ be the corresponding simple graph on the
\index{graph of the \EDM, $\GG=(N,E,\omega)$} 
nodes $N = 1\!:\!n$ whose edges $E$ correspond to the known 
entries of $D_p$, with $(D_p)_{ij}=\omega_{ij}^2$, for all $(i,j) \in E$.

\begin{figure}[htb]
\epsfxsize=310pt
\centerline{\epsfbox{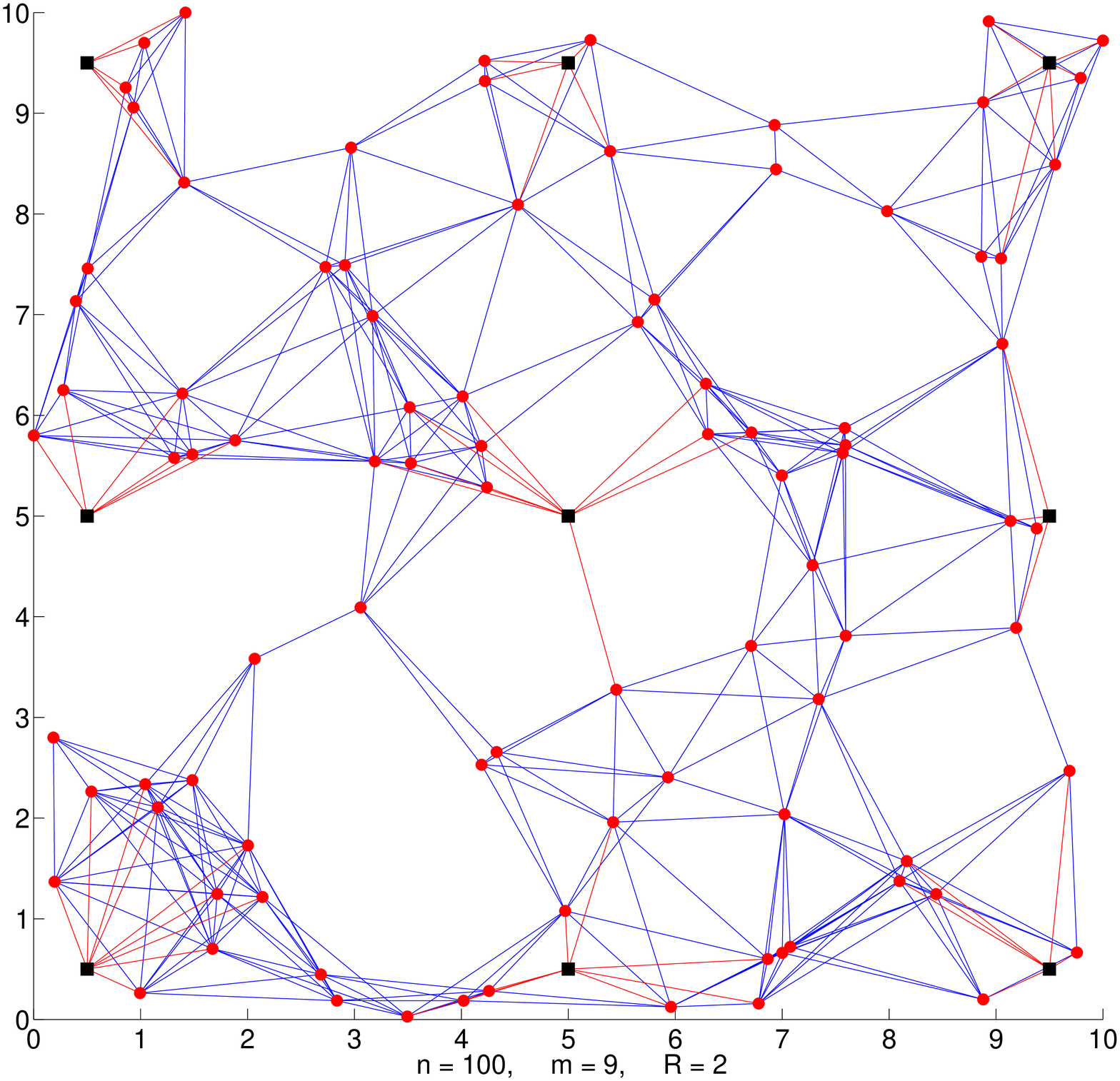}}
\caption{Graph of partial \EDM with sensors 
\textcolor{red}{$\circ$}  and anchors $\blacksquare$}
\label{fig:fc1}
\end{figure}

\begin{defi}
For $Y\in \Sn$ and $\alpha \subseteq 1\!:\!n$, we let $Y[\alpha]$ 
denote the corresponding 
\index{principal submatrix, $Y[\alpha]$}
{\em principal submatrix} formed from the rows and columns with indices
$\alpha$. If, in addition, $|\alpha |=k$ and
$\bar Y \in \Ss^k$ is given,  then we define
\[
\Sayb:= \left\{ Y \in \Sn : Y[\alpha] = \bar Y \right\}, \quad
\Sapyb:= \left\{ Y \in \Snp : Y[\alpha] = \bar Y \right\}.
\]
That is, the subset of matrices $Y \in \Sn$ ($Y \in \Snp$)
\index{principal submatrix set, $\Sayb$}
\index{principal submatrix positive semidefinite set, $\Sapyb$}
with principal submatrix $Y[\alpha]$ fixed to $\bar Y$.
\end{defi}
\noindent
For example, the subset of matrices
in $\Sn$ with the top left $k\times k$ block fixed is
\beq
\label{eq:setSblock}
\Skyb = \left\{Y \in \Sn : Y= 
\left[
\begin{array}{c|c}
 \bar Y & \cdot \cr
 \hline 
 \cdot & \cdot
\end{array}
\right]
\right\}.
\index{top-left block fixed, $\Skyb)$}
\index{$\Skyb)$, top-left block fixed}
\index{principal submatrix top-left block, $\Skyb$}
\index{$\Skyb$, principal submatrix top-left block}
\eeq

A clique $\gamma \subseteq 1\!:\!n$
\index{clique}
in the graph $\GG$ corresponds to a subset of sensors for  
which the distances $\omega_{ij}=\|p_i-p_j\|_2$ are known, for all
$i,j \in \gamma$; equivalently, the clique corresponds
to the principal submatrix $D_p[\gamma]$ of the
partial \EDM matrix $D_p$, where all the elements of $D_p[\gamma]$ are known.
	
Suppose that we are given a subset of the 
(squared) distances from \eqref{eq:dn} in the form of a partial \EDMc $D_p$. 
The {\em \EDM completion problem} consists of finding the missing entries 
of $D_p$ to complete the $\EDM$; see Figure \ref{fig:fc1}.
\index{\EDM completion problem}
This completion problem can be solved by finding a set of points $p_1,\ldots,p_n \in \R^r$ satisfying \eqref{eq:dn}, where $r$ is the embedding dimension of the partial \EDMc $D_p$.
\index{embedding dimension ({\em fixed}), $r$}
This problem corresponds to the graph
realizability problem with dimension $r$, which is the problem of finding 
\index{graph realizability}
positions in $\R^r$ for the vertices of a graph such that the 
inter-distances of these positions satisfy the 
given edge lengths of the graph.

Let $Y \in \Mn$ be an $n \times n$ real matrix and $y \in \Rn$ a vector.
\index{$n\times n$ matrices, $\Mn$}
We let $\mbox{diag}(Y)$ denote the vector in $\Rn$ formed from the 
diagonal of $Y$ and we let $\mbox{Diag}(y)$ denote the diagonal matrix in 
$\Mn$ with the vector $y$ along its diagonal. 
Note that $\diag$ and $\Diag$ are the adjoint 
linear transformations of each other: $\Diag = \mbox{diag}^*$.  
The operator $\offDiag$ can then be defined as 
$\offDiag (Y) := Y - \Diag ( \diag Y )$.
\index{offDiag operator of a matrix, $\offDiag M$}
\index{diagonal of a matrix, $\diag M$}
\index{diagonal matrix from a vector, $\Diag v$} 
For 
\[ 
P=\begin{bmatrix}
p_1^T\cr p_2^T\cr \vdots\cr p_n^T\cr
\end{bmatrix} \in \MM^{n\times r},
\] 
where $p_j$, $j=1,\ldots,n$, are the points used in \eqref{eq:dn},
let $Y:=PP^T$, and let $D$ be the corresponding \EDM satisfying \eqref{eq:dn}.
\index{matrix of points in space, $P$}
Defining the linear operators $\KK$ and $\DD_e$ on $\Sn$ as follows, we see that
\beq \begin{array}{rcl} \label{KK}
\KK(Y)& := &  \DD_e(Y)-2Y \\
      & := &  \mbox{diag}(Y)\,e^T + e\,\mbox{diag}(Y)^T - 2Y \\
      &  = &  \left( p_{i}^{T} p_{i} + p_{j}^{T} p_{j} 
                     - 2 p_{i}^{T} p_{j} \right)_{i,j = 1}^{n} \\
      &  = &  \left( \| p_{i} - p_{j} \|_{2}^{2} \right)_{i,j = 1}^{n} \\
      &  = &  D.
\end{array} \eeq
That is, $\KK$ maps the positive semidefinite matrix $Y$ onto the \EDM D.
More generally, we can allow for a general vector $v$ to replace $e$,
and define $\DD_v(Y) := \mbox{diag}(Y)\,v^T + v\,\mbox{diag}(Y)^T$. 
By abuse of notation, we also allow $\DD_v$ to act on a vector;
that is, $\DD_v(y) := yv^T+vy^T$.
The adjoint of $\KK$ is
\beq \begin{array}{rcl} \label{KKs}
\KK^*(D) &=& 2(\Diag(De)-D).
\end{array} \eeq

The linear operator $\KK$ is one-one and onto between the 
{\em centered} and {\em hollow} subspaces of $\Sn$, which are defined as
\beq \begin{array}{rcll}
\Sc &:=&  \{ Y \in \Sn :  Ye = 0 \} & \mbox{(zero row sums)}, \\
\Sh& := & \{ D \in \Sn :  \diag(D) = 0 \} &  = \RR(\offDiag).
\end{array}
\eeq
Let $J:=I-\frac 1n ee^T$ denote the orthogonal projection onto the
\index{$J$, orthogonal projection onto $\{e\}^\perp$}
subspace $\{e\}^\perp$ and define the linear operator
$\TT (D) := -\frac 12 J \offDiag (D) J$. 
\index{$\TT = \KK^\dagger$}
Then we have the following relationships.
\begin{prop}(\cite{homwolkA:04})
\label{prop:KTonetoone}
The linear operator $\TT$ is the generalized inverse of 
the linear operator $\KK$; that is, $\KK^\dagger =\TT$.  
Moreover:
\beq \label{eq:RN1}
\begin{array}{rcl}
\RR (\KK) = \Sh; & \quad  &\NN(\KK) = \RR(\DD_e);\\
\RR(\KK^*)=\RR (\TT) = \Sc; & \quad & \NN(\KK^*)= \NN(\TT) = \RR(\Diag);
\end{array}
\eeq
\beq \label{eq:RN3}
\begin{array}{rcl}
\Sn= \Sh \oplus  \RR(\Diag) = \Sc \oplus \RR(\DD_e).
\end{array}
\eeq
\end{prop}

\begin{thm}(\cite{homwolkA:04})
\label{thm:KTonetoone}
The linear operators $\TT$ and $\KK$ are one-to-one and onto mappings between the cone $\En \subset \Sh$ and the face of the semidefinite cone $\Snp \! \cap \Sc$.  That is, 
\[ \TT(\En) = \Snp \! \cap \Sc 
   \quad \mbox{and} \quad 
   \KK(\Snp \! \cap \Sc) = \En. \]
\end{thm}

\begin{rem}
$D \in \En$ has embedding dimension $r$ if and only if
\index{embedding dimension ({\em fixed}), $r$}
$\KK^\dagger (D)\succeq 0$ and
$\rank\!(\KK^\dagger (D))=r$. In addition, we get $\KK^\dagger (D) e=0$.
Therefore, we can factor $\KK^\dagger (D) =PP^T$, 
for some $P \in \MM^{n\times r}$, to recover the 
(centered) sensors in $\R^r$ from the rows in $P$.
Note that rotations of the points in the rows of $P$ do not change the
value $Y=PP^T$, since $PP^T=PQ^TQP$ if $Q$ is orthogonal.
However, the nullspace of $\KK$ is related to translations of the points in $P$.
Let $D \in \En$ with embedding dimension $r$ and 
let $Y:=\KK^{\dagger}(D)$ have full rank factorization $Y=PP^T$, with $P
\in \MM^{n \times r}$.
Then the translation of points in the rows of $P$ to $\bar P:=P+ew^T$, for 
some $w\in \R^r$, results in $\bar Y := \bar P \bar P^T = Y+\DD_e(y)$, with
$y:=Pw+\frac {w^Tw}2 e$, and 
$\KK(\bar Y)=\KK(Y)=D$, since $\DD_e(y)\in \NN(\KK)$.
Note that $\RR(Y)=\RR(P)$, therefore 
$y=Pw+\frac {w^Tw}2 e \in \RR(Y) + \cone \{e\}$,
as we will also see in more generality in Lemma~\ref{lem:psdDe} below.
\end{rem}

Let $D_p \in \Ss^n$ be a \emph{partial} \EDM with embedding dimension $r$
and let $W \in \Ss^n$ be the $0$--$1$ matrix corresponding to the known entries of $D_p$.
One can use the substitution $D=\KK(Y)$, where $Y\in \Snp \cap \Ss_C$,
in the \EDM completion problem 
\[
	\begin{array}{cc}
	\mbox{Find} & D \in \EE^n \\
	\mbox{s.t.} & W \circ D = W \circ D_p
	\end{array}
\]
to obtain the \SDP relaxation
\[
	\begin{array}{cc}
	\mbox{Find} & Y \in \Ss^n_+ \cap \Ss_C \\
	\mbox{s.t.} & W \circ \KK(Y) = W \circ D_p
	\end{array}.
\]
This relaxation does not restrict the rank of $Y$ and may
yield a solution with embedding dimension that is too large, 
if $\rank (Y) > r$.
Moreover, solving \SDP problems with rank restrictions is NP-hard.
However, we work on faces of $\Snp$ described by $U\Ss_+^tU^T$, with $t \leq n$.
In order to find the face with the smallest dimension $t$, we must have the correct knowledge of the matrix $U$. 
In this paper, we obtain information on $U$ using the cliques in the 
graph of the partial \EDM.

\section{Semidefinite Facial Reduction}
\label{sect:cliqred}

We now present several techniques for reducing an \EDM completion problem when
one or more (possibly intersecting) cliques are known.
This extends the reduction using disjoint cliques presented in
\cite{DiKrQiWo:06,DiKrQiWo:08}. In each case, we take advantage of the
loss of Slater's constraint qualification and project the problem
to a lower dimensional \SDP cone.

We first need the following two technical lemmas 
that exploit the structure of the \SDP cone.
\begin{lem}
\label{lem:psdDe}
Let $B \in \Sn$, $Bv=0$, $v \neq 0$, $y \in \Rn$ and $\bar Y := B+\DD_v(y)$.
If $\bar Y \succeq 0$, then 
\[
        y \in \RR(B) + \cone \{v\}.
\]
\end{lem}
\begin{proof}
First we will show that $y \in \RR(B) + \sspan\{v\} = \RR\left( \begin{bmatrix} B & v \end{bmatrix} \right)$.  If this is not the case, then $y$ can be written as the orthogonal decomposition 
\[ 
	y = B u + \beta v + \bar y,
\]
where $0 \neq \bar y  \in \RR\left( \begin{bmatrix} B & v \end{bmatrix} \right)^\perp = \NN\left( \begin{bmatrix} B & v \end{bmatrix}^{T} \right)$.  
Note that $\bar y$ satisfies $B \bar y = 0$ and $v^{T} \bar y = 0$.  To get a contradiction with the assumption that $\bar Y \succeq 0$, we let 
\[
	z := \frac{1}{2}\frac{v}{\|v\|^{2}} - (1+|\beta|)\frac{\bar y}{\|\bar y\|^{2}},
\]
and observe that $Bz = 0$ and $v^{T}z = 1/2$.  Then, 
\[
\begin{array}{rcl}
z^T\bar Y z 
&=&
 z^T \DD_v(y) z
\\&=&
 z^T \left( yv^T +vy^T \right) z
\\&=&
 y^{T} z
\\&=&
 \frac{1}{2}\beta + \bar y^{T} z
\\&<&
 \frac{1}{2}(1+|\beta|) + \bar y^Tz
\\&=&
 -\frac{1}{2}(1+|\beta|)
\\&<& 0,
\end{array}
\]
which gives us the desired contradiction.  Therefore, $y \in \RR(B) + \sspan\{v\}$, so to show that $y \in \RR(B) + \cone\{v\}$, we only need to show that if $y = B u + \beta v$, then $\beta \geq 0$.  First note that $v^{T}y = \beta v^{T}v$.  Then,
\[
\begin{array}{rcl}
v^T \bar Y v 
&=&
 v^T \left( yv^T +vy^T \right) v
\\&=&
 2 v^{T}y v^{T}v
\\&=&
 2 \beta (v^{T}v)^{2}.
\end{array}
\]
Since $\bar Y \succeq 0$, we have $2 \beta (v^{T}v)^{2} \geq 0$. This
implies that $\beta \geq 0$, since $v \neq 0$.
\end{proof}

If $\bar Y \in \Skp$, then we can use the minimal face of $\Skp$ 
containing $\bar Y$ to find an expression for the minimal face of 
$\Snp$ that contains $\Skpyb$.

\begin{lem}
\label{lem:rangeYbar}
Let $\bar U \in \MM^{k\times t}$ with $\bar U^T \bar U=I_t$.  
If $\face \{\bar Y\} \unlhd \bar U \Stp \bar U^{T}$, then
\beq
\label{eq:facereprYsub}
\face \Skpyb \unlhd 
\begin{bmatrix} \bar U & 0 \cr 0 & I_{n-k} \end{bmatrix} 
\Ss^{n-k+t}_+
\begin{bmatrix} \bar U & 0 \cr 0 & I_{n-k} \end{bmatrix}^T.
\eeq
Furthermore, if $ \face \{\bar Y\} = \bar U \Stp \bar U^{T}$,
then 
\beq
\label{eq:facereprY}
\face \Skpyb = 
\begin{bmatrix} \bar U & 0 \cr 0 & I_{n-k} \end{bmatrix} 
\Ss^{n-k+t}_+
\begin{bmatrix} \bar U & 0 \cr 0 & I_{n-k} \end{bmatrix}^T.
\eeq
\end{lem}

\begin{proof}
Since $\bar Y \in \bar U \Stp \bar U^{T}$, then $\bar Y = \bar U S \bar U^T$,
for some $S \in \Stp$.
Let $Y \in \Skpyb$ and choose $\bar V$ so that 
$\begin{bmatrix}\bar U & \bar V\end{bmatrix}$ is an orthogonal matrix. 
Then, with $Y$ blocked appropriately, we evaluate the congruence
\[
0\preceq \begin{bmatrix} \bar V & 0 \cr 0  &I_{n-k} \end{bmatrix}^T
Y \begin{bmatrix} \bar V & 0 \cr 0 & I_{n-k} \end{bmatrix} =
\begin{bmatrix} 0 & \bar V^T Y_{21}^T\cr   Y_{21} \bar V &  Y_{22} \end{bmatrix}=
\begin{bmatrix} 0 & 0 \cr   0 &  Y_{22} \end{bmatrix}.
\]
Therefore, $Y\succeq 0$ implies that $\bar V^TY_{21}^T =0$.
Since $\NN(\bar V^T)= \RR(\bar U)$, we get $Y_{21}^T=\bar U X$, for some $X$.  
Therefore, we can write 
\[
Y= 
\begin{bmatrix} \bar U & 0 \cr 0 & I_{n-k} \end{bmatrix} 
\begin{bmatrix} S & X \cr X^T &  Y_{22} \end{bmatrix}
\begin{bmatrix} \bar U & 0 \cr 0 & I_{n-k} \end{bmatrix}^T.
\]
This implies that $\face \Skpyb \unlhd U \Ss^{n-k+t}_+ U^T$, where
\[
U :=
\begin{bmatrix} \bar U & 0 \cr 0 & I_{n-k} \end{bmatrix}.
\]
This proves \eqref{eq:facereprYsub}. To prove \eqref{eq:facereprY},
note that if $ \face \{\bar Y\} = \bar U \Stp \bar U^{T}$ then
$ \bar Y \in \relint  \left( \bar U \Stp \bar U^{T}\right)$, so
$\bar Y = \bar U S \bar U^{T}$, for some $S \in \Stpp$.  Letting
\[
\hat{Y} :=
\begin{bmatrix} \bar U & 0 \cr 0 & I_{n-k} \end{bmatrix}
\begin{bmatrix} S & 0 \cr 0 & I_{n-k} \end{bmatrix}
\begin{bmatrix} \bar U & 0 \cr 0 & I_{n-k} \end{bmatrix}^{T},
\]
we see that $\hat{Y} \in \Skpyb \cap 
\relint  \left( U \Ss^{n-k+t}_+ U^T \right)$.
This implies that there is no smaller face of $\Snp$ containing 
$\Skpyb$, completing the proof.
\end{proof}

\subsection{Single Clique Facial Reduction}
\label{sect:singleclique}

If the principal submatrix $\bar D \in \EE^k$ is given, 
for index set $\alpha \subseteq 1\!:\!n$, with $|\alpha| = k$,
we define 
\beq
\label{eq:setS}
\Eayb:= \left\{ D \in \En : D[\alpha] = \bar D \right\}.
\eeq
\index{$\Eayb$, principal submatrix of \EDM}
\index{principal submatrix of \EDM, $\Eayb$}
Similarly, the subset of matrices
in $\En$ with the top left $k\times k$ block fixed is
\beq
\label{eq:setEblock}
\Ekyb = \left\{D \in \En : D = 
\left[
\begin{array}{c|c}
 \bar D & \cdot \cr
 \hline 
 \cdot & \cdot
\end{array}
\right]
\right\}.
\index{top-left block fixed, $\Ekyb$}
\index{principal submatrix top-left block, $\Ekyb$}
\eeq

A fixed principal submatrix $\bar D$ in a partial \EDM $D$
corresponds to a clique $\alpha$ in the graph $\GG$ of the 
partial \EDM $D$.
Given such a fixed clique defined by the submatrix $\bar D$, 
the following theorem shows that the following set, containing the feasible set 
of the corresponding \SDP relaxation,
\[
\left\{ 
Y \in \Snp \cap \Sc : 
\KK(Y[\alpha]) = \bar D 
\right\} 
= 
\KK^\dagger \left(\Eayb \right),
\]
is contained in a proper face of $\Snp$.
This means that the Slater
constraint qualification (strict feasibility) fails, and we can reduce
the size of the \SDP problem; see \cite{DiKrQiWo:06}.
We expand on this and find an explicit expression for
$\face \KK^\dagger \left(\Eayb \right)$ in the following Theorem~\ref{thm:onecliquered}.
For simplicity, here and below, we often work with
ordered sets of integers for the two cliques. This simplification
can always be obtained by a permutation of the indices of the
sensors.

\begin{thm}
\label{thm:onecliquered}
Let $D \in \En$, with embedding dimension $r$. 
Let $\bar D := D[1\!:\!k] \in \Ek$ with embedding dimension $t$,
and $B := \KK^\dagger(\bar D) = \bar U_B S \bar U_B^T$,
where $\bar U_B \in \MM^{k\times t}$, $\bar U_B^T \bar U_B = I_t$,
and $S \in \Ss^t_{++}$. 
Furthermore, let $U_B := 
\begin{bmatrix} 
\bar U_B & \frac 1{\sqrt{k}} e 
\end{bmatrix}
\in \MM^{k\times (t+1)}$,  
$U := \begin{bmatrix} U_B & 0 \cr 0 & I_{n-k}  \end{bmatrix}$,
and let $\begin{bmatrix} V & \frac{U^Te}{\|U^Te\|}  \end{bmatrix}
\in \MM^{n-k+t+1}$ be orthogonal.
Then
\beq
\label{eq:UBY}
\face \KK^\dagger\left(\Ekyb\right) 
=\left( U \Ss_+^{n-k+t+1} U^T \right) \cap \Sc
= (U V) \Ss_+^{n-k+t} (U V)^T.
\eeq
\end{thm}
\begin{proof}
Let $Y \in \KK^\dagger\left(\Ekyb\right)$ and $\bar Y := Y[1\!:\!k]$.
Then there exists $D \in \Ekyb$ such that $Y = \KK^\dagger(D)$, 
implying that $\KK(Y) = D$, and that $\KK(\bar Y) = \bar D = \KK(B)$.
Thus, $\bar Y \in B + \NN(\KK) = B + \RR(\DD_e)$, 
where the last equality follows from Proposition~\ref{prop:KTonetoone}.  
This implies that $\bar Y = B + \DD_{e}(y)$, for some $y \in \Rk$.
From Theorem~\ref{thm:KTonetoone}, we get $\bar Y \succeq 0$ and $Be=0$.
Therefore, Lemma~\ref{lem:psdDe} implies that $y = Bu + \beta e$, 
for some $u \in \Rk$ and $\beta \geq 0$. This further implies
\[
\bar Y = B + Bue^T + eu^TB + 2\beta ee^T.
\]
From this expression for $\bar Y$, we can see that
$\RR(\bar Y) \subseteq 
\RR \left(\begin{bmatrix} B & e \end{bmatrix}\right) 
= \RR(U_B)$, where the last equality follows from the fact that $Be = 0$.
Therefore, $\bar Y \in U_{B} \Ss^{t+1}_+ U_{B}^{T}$, 
implying, by Lemma~\ref{lem:rangeYbar}, that 
$\face \Skpyb \unlhd U \Ss^{n-k+t+1}_+ U^T$.
Since $Y \in \Skpyb$ and $Ye = 0$, 
we have that $Y \in \left( U \Ss_+^{n-k+t+1} U^T \right) \cap \Sc$.  
Therefore, $\face \KK^\dagger\left(\Ekyb\right) 
\unlhd \left( U \Ss_+^{n-k+t+1} U^T \right) \cap \Sc$.
Since $V^T U^T e = 0$, we have that 
\beq
\label{eq:UVdotUVT}
\left( U \Ss_+^{n-k+t+1} U^T \right) \cap \Sc 
= U V \Ss_+^{n-k+t} V^T U^T.
\eeq
To show that $\face \KK^\dagger\left(\Ekyb\right) 
= \left( U \Ss_+^{n-k+t+1} U^T \right) \cap \Sc$, we need to find 
\beq
\label{eq:Yhat}
\hat Y = UZU^T \in \KK^\dagger\left(\Ekyb\right), \quad
\mbox{with} \ \rank(\hat Y) = n-k+t, \ \hat Y e = 0, \ Z \in \Ss_+^{n-k+t+1}.
\eeq
To accomplish this, we let $T_1= \begin{bmatrix} S & 0 \cr 0 & 1 \end{bmatrix}$. 
Then $T_1 \succ 0$ and
\[
B+\frac 1k ee^T = U_B T_1 U_B^T = \bar P \bar P^T, \quad 
\mbox{where} \ \bar P := U_B T_1^{1/2} \in \MM^{k \times (t+1)}.
\]
Let
\[
P :=
\left[
\begin{array}{c|c}
\bar P & 0 \cr
  \hline
0 & I_{n-k-1} \cr
-e^T \bar P & -e^T
\end{array}
\right]
\in \MM^{n \times (n-k+t)}.
\]
Since $\bar P$ has full-column rank, we see that $P$ also has full-column rank. Moreover, $P^Te = 0$.  Therefore, 
\[
\hat Y := PP^T = 
\left[
\begin{array}{c|cc}
\bar P \bar P^T & 0 & -e \cr
  \hline
0 & I_{n-k-1} & -e \cr
-e^T & -e^T & n - 1
\end{array}
\right]
\in \Snp,
\]
satisfies $\hat Y e = 0$ and $\rank(\hat Y) = n-k+t$.  
Furthermore, we have that $\hat Y = U Z U^T$, where
\[
Z = 
\left[
\begin{array}{cc|cc}
S & 0 & 0 & 0 \cr
0 & 1 & 0 & -\sqrt{k} \cr
 \hline 
0 & 0 & I_{n-k-1} & -e \cr
0 & -\sqrt{k} & -e^T & n - 1
\end{array}
\right]
\in \Ss^{n-k+t+1}.
\]
Note that we can also write $Z$ as
\[
Z = 
\begin{bmatrix}
S & 0 \cr
0 & T 
\end{bmatrix}
\in \Ss^{n-k+t+1},
\]
where
\[
T := 
\begin{bmatrix}
1 & 0 & -\sqrt{k}  \cr
0 & I_{n-k-1} & -e \cr
-\sqrt{k} & -e^T & n - 1
\end{bmatrix}
\in \Ss^{n-k+1}.
\]
The eigenvalues of $T$ are $0$, $1$, and $n$, with multiplicities $1$, $n-k-1$, and $1$, respectively.  Therefore, $\rank(T) = n-k$, which implies that $\rank(Z) = n-k+t$ 
and $Z \succeq 0$.

Letting $\hat D := \KK(\hat Y)$, we have that $\hat D \in \Ekyb$, since
\[
\hat D[1\!:\!k] = \KK(\hat Y[1\!:\!k]) = \KK(\bar P \bar P^T) 
= \KK\left(B + \frac{1}{k}ee^T\right) = \KK(B) = \bar D.
\]
Therefore, $\hat Y$ satisfies \eqref{eq:Yhat}, completing the proof.
\end{proof}

\begin{remark}
Theorem~\ref{thm:onecliquered} provides a reduction in the dimension of
the \EDM completion problem. Initially, our problem consists in finding
$Y\in \Snp \cap \Sc$ such that the constraint
\[
\KK(Y[\alpha])=D[\alpha], \quad \alpha = 1\!:\!k,
\]
holds. 
After the reduction, we have the smaller dimensional variable
$Z \in \Ss_+^{n-k+t}$; by construction $Y := (UV)Z(UV)^T$ will automatically 
satisfy the above constraints.
This is a reduction of $k-t-1=(n-1) -(n-k+t)$ in the dimension of the matrix variable.
The addition of the vector $e$ to the range of $B$,
$U_B := \begin{bmatrix} \bar U_B & \frac 1{\sqrt{k}} e \end{bmatrix}$, 
has a geometric interpretation.
If $B=PP^T$, $P\in \MM^{k \times t}$, then the rows of $P$ provide
{\em centered} positions for the $k$ sensors in the clique $\alpha$.
However, these sensors are not necessarily centered once they are combined
with the remaining $n-k$ sensors. Therefore, we have to
allow for translations, e.g.
to $P+ev^T$ for some $v$. The multiplication
$(P+ev^T)(P+ev^T)^T=PP^T +Pve^T+ev^TP^T+ev^Tve^T$ is included in the set of matrices that we get after adding $e$ to the range of $B$.
Note that $Pve^T+ev^TP^T+ev^Tve^T = \DD_e(y)$, for $y = Pv + \frac12 ev^Tv$.
\end{remark}

The special case $k=1$ is of interest.
\begin{cor}
\label{cor:cliqueredk1}
Suppose that the hypotheses of Theorem~\ref{thm:onecliquered} hold but
that $k=1$ and $\bar D = 0$. Then $U_B=1$, $U=I_n$, and
\beq
\label{eq:UBYk1}
\face \KK^\dagger\left(\Ekyb\right) 
=\face \KK^\dagger\left(\En\right) 
= \Ss_+^{n}  \cap \Sc = V \Ss_+^{n-1} V^T,
\eeq
where  $\begin{bmatrix} V & \frac{1}{\sqrt{n}}e  \end{bmatrix}
\in \MM^{n}$ is orthogonal.
\end{cor}
\begin{proof}
Since $k=1$, necessarily we get $t=0$ and we can set $U_B=1$.
\end{proof}

\subsubsection{Disjoint Cliques Facial Reduction}
\label{sect:disjointcliques}

Theorem~\ref{thm:onecliquered} can be easily extended to two or more
disjoint cliques; see also \cite{DiKrQiWo:06}.
\begin{cor}
\label{cor:disjcliques}
Let $D \in \En$ with embedding dimension $r$. 
Let $k_0:=1< k_1 < \ldots < k_l \leq n$. 
For $i=1,\ldots, l$, let
$\bar D_i := D[k_{i-1}\!:\!k_i] \in \EE^{k_i-k_{i-1}+1}$
with embedding dimension $t_i$,
$B_i := \KK^\dagger(\bar D_i) = \bar U_{B_i} S \bar U_{B_i}^T$,
where $\bar U_{B_i} \in \MM^{k\times t_i}$, 
$\bar U_{B_i}^T \bar U_{B_i} = I_{t_i}$, $S_i \in \Ss^{t_i}_{++}$,
and $U_{B_i} := 
\begin{bmatrix} 
\bar U_{B_i} & \frac 1{\sqrt{k_i}} e 
\end{bmatrix}
\in \MM^{k\times (t_i+1)}$.
Let 
\[
U := \begin{bmatrix} U_{B_1} & \ldots & 0 & 0 \cr 
                 \vdots & \ddots & \vdots & \vdots \cr
                 0 & \ldots & U_{B_l} & 0 \cr 
                  0 &\ldots & 0& I_{n-k_l}  \end{bmatrix}
\]
and $\begin{bmatrix} V & \frac{U^Te}{\|U^Te\|}  \end{bmatrix}
\in \MM^{n-k_l+\sum_{i=1}^l t_i +l}$ be orthogonal.
Then
\beq
\label{eq:UBYdisjointcl}
\begin{array}{rcl}
\bigcap_{i=1}^l \face \KK^\dagger\left(\En(k_{i-1}\!:\!k_i,\bar D_i)   \right) 
&=&
\left( U \Ss_+^{n-k_l+\sum_{i=1}^l t_i +l} U^T \right) \cap \Sc
\\&=&
 (U V) \Ss_+^{n-k_l+\sum_{i=1}^l t_i+l-1} (U V)^T.
\end{array}
\eeq
\end{cor}
\begin{proof}
The result follows from noting that the range of $U$ is the intersection
of the ranges of the matrices $U_{B_i}$ with appropriate identity
blocks added.
\end{proof}

\subsection{Two (Intersecting) Clique Facial Reduction}
\label{sect:twoclique}
The construction \eqref{eq:UVdotUVT} illustrates how we can find the
intersection of two faces. Using this approach, we now extend 
Theorem~\ref{thm:onecliquered} to two cliques that (possibly) intersect;
see the ordered indices in \eqref{eq:ordcliques} 
and the corresponding Venn diagram in Figure~\ref{fig:alpha}.
We also find expressions for the intersection of
the corresponding faces in $\Snp$; see equation \eqref{eq:UBY2}.
The key is to find the intersection of the subspaces that 
represent the faces, as in condition \eqref{eq:rangeU}.

\begin{figure}[htb]
\epsfxsize=180pt
\centerline{\epsfbox{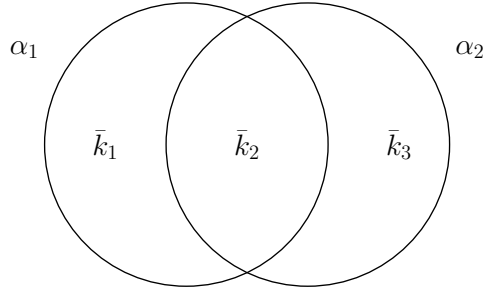}}
\caption{Venn diagram of the sets of ordered indices, $\alpha_1$ and $\alpha_2$, 
in Theorem~\ref{thm:interclique}}
\label{fig:alpha}
\end{figure}

\begin{thm}
\label{thm:interclique}
Let $D \in \En$ with embedding dimension $r$ and, as in Figure~\ref{fig:alpha}, 
define the sets of positive integers
\beq
\label{eq:ordcliques}
\begin{array}{c}
 \alpha_1  :=  1\!:\!(\bar k_1 + \bar k_2), \quad
  \alpha_2  :=  (\bar k_1 + 1)\!:\!(\bar k_1 + \bar k_2 + \bar k_3)
   \subseteq  1\!:\!n,  \\
 k_1  :=  |\alpha_1| = \bar k_1 + \bar k_2, \quad
  k_2  :=  |\alpha_2| = \bar k_2 + \bar k_3, \\
k := \bar k_1 + \bar k_2 + \bar k_3.
\end{array}
\eeq
For $i=1,2$, let $\bar D_i := D[\alpha_i] \in \EE^{k_i}$
with embedding dimension $t_i$,
and $B_i := \KK^\dagger(\bar D_i) = \bar U_i S_i \bar U_i^T$,
where $\bar U_i \in \MM^{k_i\times t_i}$, $\bar U_i^T \bar U_i = I_{t_i}$,
$S_i \in \Ss^{t_i}_{++}$, and
$U_i := 
\begin{bmatrix} 
\bar U_i & \frac 1{\sqrt{k_i}} e 
\end{bmatrix}
\in \MM^{k_i \times (t_i+1)}$.
Let $t$ and $\bar U \in \MM^{k \times (t+1)}$ satisfy
\beq
\label{eq:rangeU}
\RR(\bar U) = 
\RR \left(
\begin{bmatrix} U_1 & 0 \cr 0 & I_{\bar k_3}  \end{bmatrix}
\right)
\cap
\RR \left(
\begin{bmatrix} I_{\bar k_1} & 0 \cr 0 & U_2 \end{bmatrix}
\right), \mbox{  with  }  \bar U^T \bar U = I_{t+1}.
\eeq
Let
$U := \begin{bmatrix} \bar U & 0 \cr 0 & I_{n-k} \end{bmatrix}
  \in \MM^{n \times (n-k+t+1)}$
and 
$\begin{bmatrix} V & \frac{U^Te}{\|U^Te\|} \end{bmatrix}
\in \MM^{n-k+t+1}$ be orthogonal. Then
\beq
\label{eq:UBY2}
\bigcap_{i=1}^2
\face \KK^\dagger\left(\En(\alpha_i,\bar D_i)\right)
=\left( U \Ss_+^{n-k+t+1} U^T \right) \cap \Sc
= (U V) \Ss_+^{n-k+t} (U V)^T.
\eeq
\end{thm}

\begin{proof}
From Theorem~\ref{thm:onecliquered}, we have that
\[
  \face \KK^\dagger\left(\En(\alpha_1,\bar D_1)\right) = 
  \left(
  \left[\begin{array}{cc|c}
    U_1 & 0 & 0 \\
    0 & I_{\bar k_3} & 0 \\
    \hline
    0 & 0 & I_{n-k}
  \end{array}\right]  
  \Ss_+^{n-k_1+t_1+1}
  \left[\begin{array}{cc|c}
    U_1 & 0 & 0 \\
    0 & I_{\bar k_3} & 0 \\
    \hline
    0 & 0 & I_{n-k}
  \end{array}\right]^T
  \right) \cap \Sc
\]
and, after a permutation of rows and columns in 
Theorem~\ref{thm:onecliquered},
\[
  \face \KK^\dagger\left(\En(\alpha_2,\bar D_2)\right) = 
  \left(
  \left[\begin{array}{cc|c}
    I_{\bar k_1} & 0 & 0 \\
    0 & U_2 & 0 \\
    \hline
    0 & 0 & I_{n-k}
  \end{array}\right]
  \Ss_+^{n-k_2+t_2+1}
  \left[\begin{array}{cc|c}
    I_{\bar k_1} & 0 & 0 \\
    0 & U_2 & 0 \\
    \hline
    0 & 0 & I_{n-k}
  \end{array}\right]^T
  \right) \cap \Sc.
\]
The range space condition \eqref{eq:rangeU} then implies that
\[
\RR(U) = 
\RR \left(
\left[\begin{array}{cc|c}
  U_1 & 0 & 0 \\
  0 & I_{\bar k_3} & 0 \\
  \hline
  0 & 0 & I_{n-k}
\end{array}\right]  
\right)
\cap
\RR \left(
\left[\begin{array}{cc|c}
  I_{\bar k_1} & 0 & 0 \\
  0 & U_2 & 0 \\
  \hline
  0 & 0 & I_{n-k}
\end{array}\right]
\right),
\]
giving us the result \eqref{eq:UBY2}.
\end{proof}

\begin{remark}
Theorem~\ref{thm:interclique} provides a reduction in the dimension of
the \EDM completion problem. 
Initially, our problem consists in finding
$Y\in \Snp \cap \Sc$ such that the two constraints
\[
\KK(Y[\alpha_i])=D[\alpha_i],\quad i=1,2,
\]
hold. After the reduction, we want to find the smaller dimensional
$Z \in \Ss_+^{n-k+t}$; by construction $Y := (UV)Z(UV)^T$ will automatically 
satisfy the above constraints.
\end{remark}

The explicit expression for the intersection of the two faces is
given in equation~\eqref{eq:UBY2} and uses the matrix $\bar U$
obtained from the intersection of the two ranges 
in condition~\eqref{eq:rangeU}. 
Finding a matrix whose range is the intersection
of two subspaces can be done using
\cite[Algorithm~12.4.3]{GoVan:79}. However, our subspaces have special
structure. We can exploit this structure to find the intersection;
see Lemma~\eqref{lem:twocliquered} and Lemma~\eqref{lem:twocliquereddeg} below.

The dimension of the face in \eqref{eq:UBY2} is reduced to $n-k+t$.
However, we can get a dramatic reduction if we have a common
block with embedding dimension $r$, and a reduction
in the case the common block has embedding dimension $r-1$ as well. 
This provides an algebraic
proof using semidefinite programming of the rigidity of the union 
\index{rigidity}
of the two cliques under this intersection assumption.

\subsubsection{Nonsingular Facial Reduction with Intersection Embedding Dimension $r$}
\label{sect:rigidreduction}
\index{nonsingular reduction}

We now consider the case when the intersection of the two cliques
results in $D[\alpha_1\cap \alpha_2]$ having embedding dimension $r$;
see Figure~\ref{fig:intersembedr}.
\begin{figure}[htb]
\epsfxsize=180pt
\centerline{\epsfbox{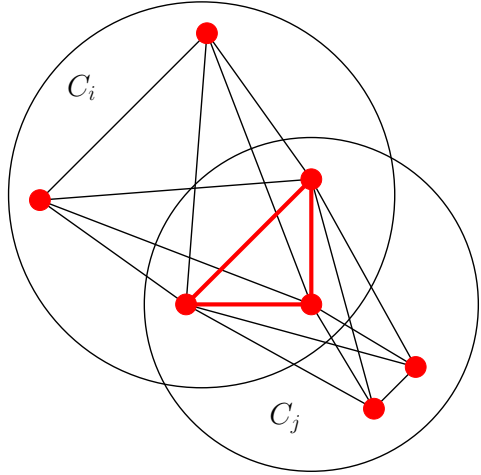}}
\caption{Two clique reduction with intersection with embedding 
       dimension $r$}
\label{fig:intersembedr}
\end{figure}
We see that we can explicitly
find the completion of the \EDM $D[\alpha_1\cup \alpha_2]$.
We first need the following result on the
intersection of two structured subspaces.

\begin{lem}
\label{lem:twocliquered}
Let
\[
U_1:= \kbordermatrix{ 
       & r+1 \cr
  s_1  & U^{\prime}_1  \cr 
  k    & U^{\prime\prime}_1 
}, \quad
U_2:= \kbordermatrix{
      & r+1 \cr
  k   & U^{\prime\prime}_2 \cr 
  s_2 & U^{\prime}_2
}, \quad
\hat U_1:= \kbordermatrix{
      & r+1  & s_2 \cr
  s_1 & U_1'  & 0   \cr 
  k   & U_1'' & 0   \cr 
  s_2 & 0     & I
}, \quad
\hat U_2:= \kbordermatrix{ 
      & s_1 & r+1  \cr
  s_1 & I   & 0 \cr 
  k   & 0   & U_2'' \cr
  s_2 & 0   & U_2'
}
\]
be appropriately blocked with $U_1'', U_2'' \in \MM^{k \times (r+1)}$
full column rank and $\RR(U_1'') = \RR(U_2'')$.
Furthermore, let
\beq
\label{eq:U1U2}
\bar U_1 := 
\kbordermatrix{ 
      & r+1 \cr
  s_1 & U_1^\prime\cr
  k   & U_1^{\prime\prime}\cr
  s_2 & U_2^{\prime} (U_2^{\prime\prime})^\dagger U_1^{\prime\prime}
}, \quad
\bar U_2 := 
\kbordermatrix{ 
      & r+1 \cr
  s_1 & U_1^{\prime} (U_1^{\prime\prime})^\dagger U_2^{\prime\prime}\cr
  k   & U_2^{\prime\prime}\cr
  s_2 & U_2^\prime
}.
\eeq
Then $\bar U_1$ and $\bar U_2$ are full column rank and satisfy
\[
\RR (\hat U_1) \cap  \RR (\hat U_2) =   
\RR\left(\bar U_1 \right) = \RR\left(\bar U_2 \right).
\]
Moreover, if $e_{r+1} \in \R^{r+1}$ is the $(r+1)^\mathrm{st}$ standard unit vector, and
$U_ie_{r+1} = \alpha_i e$, for some $\alpha_i \neq 0$, for $i=1,2$, then
$\bar U_ie_{r+1} = \alpha_i e$, for $i=1,2$.
\end{lem}
\begin{proof}
\label{pageproofblocks}
From the definitions, $x \in  \RR(\hat U_1) \cap \RR(\hat U_2)$
if and only if
\[
  x = 
  \begin{bmatrix}
    x_1 \\
    x_2 \\
    x_3 
  \end{bmatrix}
  =
  \begin{bmatrix}
	  U^{\prime}_1  v_1 \\
    U^{\prime\prime}_1 v_1 \\
          v_2 
  \end{bmatrix}
  =
  \begin{bmatrix}
          w_1 \\
	  U^{\prime\prime}_2 w_2 \\
	  U^{\prime}_2  w_2 
  \end{bmatrix}, \quad
  \mbox{for some 
  $v = \begin{bmatrix} v_1 \cr v_2 \end{bmatrix}$, 
  $w = \begin{bmatrix} w_1 \cr w_2 \end{bmatrix}$. }
\]
Note that $U_1'' v_1 = U_2'' w_2$ if and only if $w_2 = (U_2'')^\dagger U_1'' v_1$;
this follows from the facts that
$U_2''$ full column rank implies $(U_2'')^\dagger U_2'' = I$,
and 
$\RR(U_1'') = \RR(U_2'')$ implies $U_2''(U_2'')^\dagger U_1'' = U_1''$.
Therefore, $x \in  \RR(\hat U_1) \cap \RR(\hat U_2)$ if and only if
\[
  x = 
  \begin{bmatrix}
    x_1 \\
    x_2 \\
    x_3 
  \end{bmatrix}
  =
  \begin{bmatrix}
	  U_1'  v_1 \\
    U_1'' v_1 \\
  U_2' (U_2'')^\dagger U_1'' v_1
  \end{bmatrix}
  = 
  \bar U_1 v_1,
  \quad \mbox{for some $v_1$},
\]
with $v_2 := U_2' (U''_2)^\dagger U''_1 v_1$, $w_1 := U_1'v_1$, and 
$w_2 := (U''_2)^\dagger U''_1 v_1$, implying that
$\RR(\hat U_1) \cap \RR(\hat U_2) = \RR(\bar U_1)$;
a similar argument shows that $\RR(\hat U_1) \cap \RR(\hat U_2) = \RR(\bar U_2)$.

Now suppose, for $i=1,2$, that $U_i e_{r+1}  = \alpha_i e$, for some $\alpha_i \neq 0$.
Then $e \in \RR(\hat U_1) \cap \RR(\hat U_2)$, so $e \in \RR(\bar U_1)$, 
implying that $\bar U_1 v = e$, for some vector $v$.
Since $\bar U_1 = \begin{bmatrix} U_1 \cr U_2' (U_2'')^\dagger U_1'' \end{bmatrix}$,
we have $U_1 v = e$. 
Furthermore, since $U_1$ has full column rank, we conclude that 
$v = \frac{1}{\alpha_1} e_{r+1}$, implying that $\bar U_1 e_{r+1} = \alpha_1 e$.
Similarly, we can show that $\bar U_2 e_{r+1} = \alpha_2 e$.
\end{proof}

We now state and prove a key result that shows we can complete the distances in the
union of two cliques provided that their intersection has embedding dimension equal to $r$.

\begin{thm}
\label{thm:twocliquered}
Let the hypotheses of Theorem~\ref{thm:interclique} hold.  
Let 
\[
  \beta \subseteq \alpha_1 \cap \alpha_2, \quad 
  \bar D := D[\beta], \quad
  B := \KK^\dagger(\bar D), \quad \bar U_\beta := \bar U[\beta,:],  
\]
where 
$\bar U  \in \MM^{k \times (t+1)}$
satisfies equation~\eqref{eq:rangeU}.
Let 
$\begin{bmatrix} \bar V & \frac{\bar U^Te}{\|\bar U^Te\|} \end{bmatrix}
\in \MM^{t+1}$ be orthogonal.  Let 
\beq
\label{eq:Zcalc1}
Z:= (J \bar U_\beta \bar V)^\dagger B ((J \bar U_\beta \bar V)^\dagger)^T. 
\eeq
If the embedding dimension for $\bar D$ is $r$,  then 
$t = r$, $Z \in \Ss^{r}_{++}$ is the unique 
solution of the equation
\begin{equation}
  \label{eq:Zbeta}
  (J \bar U_\beta \bar V) Z (J \bar U_\beta \bar V)^T = B,
\index{$J$, orthogonal projection onto $\{e\}^\perp$}
\end{equation}
and
\beq
\label{eq:DKKUVZ}
D[\alpha_1 \cup \alpha_2] = \KK\left((\bar U \bar V)Z(\bar U \bar V)^T\right).
\eeq
\end{thm}

\begin{proof}
Since the embedding dimension of $\bar D$ is $r$, we have $\rank(B) = r$.  
Furthermore, we have $Be = 0$ and $B \in \Ss^{|\beta|}_+$, implying that
$|\beta| \geq r+1$.  In addition, since the embedding dimension of $D$ is also
$r$, we conclude that the embedding dimension of $\bar D_i$ is $r$, for $i=1,2$.
Similarly, the embedding dimension of $D[\alpha_1 \cap \alpha_2]$ is also $r$.

Since $\bar U \in \MM^{k \times (t+1)}$ satisfies equation~\eqref{eq:rangeU},
we have that
\[
\RR(\bar U) = 
\RR \left(
\begin{bmatrix} 
  U^{\prime}_1 & 0 \cr 
  U^{\prime\prime}_1 & 0 \cr 
  0 & I_{\bar k_3} 
\end{bmatrix}
\right)
\cap
\RR \left(
\begin{bmatrix} 
  I_{\bar k_1} & 0 \cr 
  0 & U^{\prime\prime}_2 \cr 
  0 & U^{\prime}_2
\end{bmatrix}
\right).
\]
Note that we have partitioned $U_i  = 
\begin{bmatrix} \bar U_i & \frac{1}{\sqrt{k_i}} e \end{bmatrix}
\in \MM^{k_i \times (r+1)}$ so that $U^{\prime\prime}_i =
\begin{bmatrix} \bar U''_i & \frac{1}{\sqrt{k_i}} e \end{bmatrix}
\in \MM^{|\alpha_1 \cap \alpha_2| \times (r+1)}$, for $i=1,2$.
Moreover, we have used the fact that the embedding dimension of 
$\bar D_i$ is $r$, so that $t_i=r$, for $i=1,2$.

We claim that $U_1''$ and $U_2''$ have full column rank and that 
$\RR(U_1'') = \RR(U_2'')$.
First we let $Y := \KK^\dagger(D[\alpha_1 \cup \alpha_2])$.  
Then $Y \in \KK^\dagger\left(\Ek(\alpha_1,\bar D_1)\right)$.
By Theorem~\ref{thm:onecliquered}, there exists $Z_1 \in \Ss^{\bar k_3 + r + 1}_+$
such that
\[
  Y = 
  \begin{bmatrix} 
    U^{\prime}_1 & 0 \cr 
    U^{\prime\prime}_1 & 0 \cr 
    0 & I_{\bar k_3} 
  \end{bmatrix}
  Z_1
  \begin{bmatrix} 
    U^{\prime}_1 & 0 \cr 
    U^{\prime\prime}_1 & 0 \cr 
    0 & I_{\bar k_3} 
  \end{bmatrix}^T.
\]
Therefore, $Y[\alpha_1 \cap \alpha_2] = 
\begin{bmatrix} U^{\prime\prime}_1 & 0 \end{bmatrix}
Z_1
\begin{bmatrix} U^{\prime\prime}_1 & 0 \end{bmatrix}^T
\in
U_1'' \Ss^{r+1}_+ (U_1'')^T$, so 
\[
	\RR(Y[\alpha_1 \cap \alpha_2]) \subseteq \RR(U_1'').
\]
Furthermore, since $\KK(Y) = D[\alpha_1 \cup \alpha_2]$, we have that 
$\KK(Y[\alpha_1 \cap \alpha_2]) = D[\alpha_1 \cap \alpha_2] 
= \KK \left( \KK^\dagger( D[\alpha_1 \cap \alpha_2] ) \right)$, 
so $Y[\alpha_1 \cap \alpha_2] \in \KK^\dagger( D[\alpha_1 \cap \alpha_2] ) + \NN(\KK)$.
Since $\NN(\KK) = \RR(\DD_e)$, there exists a vector $y$ such that
\[
	Y[\alpha_1 \cap \alpha_2] = \KK^\dagger( D[\alpha_1 \cap \alpha_2] ) + \DD_e(y)
		= \KK^\dagger( D[\alpha_1 \cap \alpha_2] ) + ye^T + ey^T.
\]
By Lemma~\ref{lem:psdDe}, 
$y \in 
\RR \left( 
\begin{bmatrix} \KK^\dagger( D[\alpha_1 \cap \alpha_2] ) & e \end{bmatrix} 
\right)$.
Therefore, 
\[
	\RR(Y[\alpha_1 \cap \alpha_2])
	= 
	\RR \left( 
			\begin{bmatrix} 
				\KK^\dagger( D[\alpha_1 \cap \alpha_2] ) & e 
			\end{bmatrix} 
		\right).
\]
Moreover, $\rank \KK^\dagger( D[\alpha_1 \cap \alpha_2] ) = r$ and
$\KK^\dagger( D[\alpha_1 \cap \alpha_2] )e = 0$, so
\[
	r+1 = \dim \RR(Y[\alpha_1 \cap \alpha_2]) \leq \dim \RR(U_1'') \leq r+1.
\]
Therefore, $U_1''$ has full column rank and 
$\RR(U_1'') = \RR(Y[\alpha_1 \cap \alpha_2])$.
Similarly, we can show that $U_2''$ has full column rank and 
$\RR(U_2'') = \RR(Y[\alpha_1 \cap \alpha_2])$, so we conclude that 
$\RR(U_1'') = \RR(U_2'')$.

We now claim that $t = r$, where $\bar U \in \MM^{k \times (t+1)}$
satisfies equation~\eqref{eq:rangeU}.  
Since $U_1'', U_2'' \in \MM^{|\alpha_1 \cap \alpha_2| \times (r+1)}$ have full column rank
and $\RR(U_1'') = \RR(U_2'')$, we have by  
Lemma~\ref{lem:twocliquered} that 
$\RR(\bar U) = \RR(\bar U_1) = \RR(\bar U_2)$, where
\[
\bar U_1 := 
\begin{bmatrix} 
  U_1^\prime\cr
  U_1^{\prime\prime}\cr
  U_2^{\prime} (U_2^{\prime\prime})^\dagger U_1^{\prime\prime}
\end{bmatrix}
\quad \mbox{and} \quad
\bar U_2 := 
\begin{bmatrix} 
  U_1^{\prime} (U_1^{\prime\prime})^\dagger U_2^{\prime\prime}\cr
  U_2^{\prime\prime}\cr
  U_2^\prime
\end{bmatrix}.
\]
Therefore, 
\[
t+1 = \dim \RR(\bar U) = \dim \RR(\bar U_1) = \dim \RR(\bar U_2) = r+1,
\]
so we have $t = r$, as claimed.

Recall, $Y = \KK^\dagger(D[\alpha_1 \cup \alpha_2])$, so 
$Y \in \cap_{i=1,2} \KK^\dagger\left(\Ek(\alpha_i,\bar D_i)\right)$.
Thus, Theorem~\ref{thm:interclique} implies that there exists 
$\bar Z \in \Ss^{r}_+$ such that 
$Y = (\bar U \bar V)\bar Z(\bar U \bar V)^T$. 
Observe that $\KK(Y[\beta]) = D[\beta] = \bar D$.  
Thus,
\[ 
  \KK\left(
  (\bar U_\beta \bar V) 
  \bar Z 
  (\bar U_\beta \bar V)^T
  \right) = \bar D, 
\]
implying that
\[ 
  \KK^\dagger \KK\left(
  (\bar U_\beta \bar V) 
  \bar Z 
  (\bar U_\beta \bar V)^T
  \right) = B. 
\]
Since $\KK^\dagger \KK$ is the projection onto $\RR(\KK^*) = \Sc$, 
we have that $\KK^\dagger \KK( \cdot ) = J(\cdot)J$. 
Therefore, we have that $\bar Z$ satisfies equation~\eqref{eq:Zbeta}.  
It remains to show that equation~\eqref{eq:Zbeta} has a unique solution.  
Let $A := J \bar U_\beta \bar V \in \MM^{|\beta| \times r}$.  
Then $A \bar Z A^T = B$ and $\rank(B) = r$ implies that $\rank(A) \geq r$, 
so $A$ has full column rank.
This implies that equation~\eqref{eq:Zbeta} has a unique solution, 
and that $\bar Z = A^\dagger B (A^\dagger)^T = Z$.
Finally, since $Y = (\bar U \bar V)Z(\bar U \bar V)^T$ and 
$D[\alpha_1 \cup \alpha_2] = \KK(Y)$, 
we get equation~\eqref{eq:DKKUVZ}.
\end{proof}

The following result shows that if we know the minimal face of $\Ss^n_+$ containing 
$\KK^\dagger(D)$, and we know a small submatrix of $D$, then we can compute 
a set of points in $\R^r$ that generate $D$ by solving a small equation.  

\begin{cor}
\label{cor:finalUbar}
Let $D \in \En$ with embedding dimension $r$, and let $\beta \subseteq 1\!:\!n$.
Let $U \in \MM^{n \times (r+1)}$ satisfy 
\[
	\face \KK^\dagger\left(D\right) 
	=
	\left( U \Ss_+^{r+1} U^T \right) \cap \Sc,
\]
let $U_\beta := U[\beta,:]$, 
and let $\begin{bmatrix} V & \frac{U^Te}{\|U^Te\|} \end{bmatrix}
\in \MM^{r+1}$ be orthogonal.
If $D[\beta]$ has embedding dimension $r$, then 
\[
	(J U_\beta V) Z (J U_\beta V)^T = \KK^\dagger(D[\beta])
\]
has a unique solution $Z \in \Ss^r_{++}$, and $D = \KK(PP^T)$, 
where $P := UVZ^{1/2} \in \R^{n \times r}$.
\end{cor}
\begin{proof}
Apply Theorem~\ref{thm:twocliquered} with $\alpha_1 = \alpha_2 = 1\!:\!n$.
\end{proof}

\begin{remark}
\label{rem:effZcalc}
A more efficient way to calculate $Z$ uses the full rank factorization
\[
B=QD^{1/2}\left(QD^{1/2}\right)^T, \quad Q^TQ=I_r, \quad D\in \Ss_{++}^r.
\]
Let $C = (J \bar U_\beta \bar V)^\dagger \left(QD^{1/2}\right)$.  Then 
$Z$ in \eqref{eq:Zcalc1} can be found from $Z=CC^T$.
Note that our algorithm postpones finding $Z$ until the end where we
can no longer perform any clique reductions. At each iteration, we
compute the matrix $\bar U$ that represents the face corresponding to
the union of two cliques; $\bar U$ is chosen from 
one of $\bar U_i$, for $i=1,2$ in \eqref{eq:U1U2}. Moreover, for
stability, we maintain $\bar U^T \bar U =I$, $\bar Ue_{r+1}=\alpha e$.

For many of our test problems, we can repeatedly apply Theorem~\ref{thm:twocliquered} 
until there is only one clique left. Since each
repetition reduces the number of cliques by one, this means that there
are at most $n$ such steps.
\end{remark}

\subsubsection{Singular Facial Reduction with Intersection Embedding Dimension $r-1$}
\label{sect:nonrigidreduction}
\index{singular reduction}

\begin{figure}[htb]
\epsfxsize=160pt
\centerline{\epsfbox{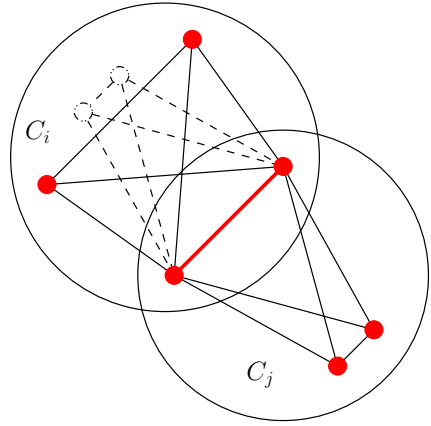}}
\caption{Two clique reduction with intersection having embedding 
       dimension $<r$}
\label{fig:deg2cliqred}
\end{figure}

We now show that if the embedding dimension of 
the intersection is $r-1$ (i.e., deficient), 
then we can find at most two completions. 
If exactly one of these two completions is feasible in the sense that it satisfies the related distance equality constraints and, if included, the related lower bound inequality constraints obtained from the radio range $R$, then we have identified the unique completion; see Figure~\ref{fig:deg2cliqred}. 
We first need the following extension of Lemma~\ref{lem:twocliquered} 
on the intersection of two structured
subspaces for the case where the common middle blocks are not full rank.
\index{singular intersection}

\begin{lem}
\label{lem:twocliquereddeg}
Let $U_i, \hat U_i, \bar U_i$, for $i=1,2$, be defined and appropriately blocked
as in Lemma~\ref{lem:twocliquered},
with $U^{\prime\prime}_i \in \MM^{k \times (r+1)}$
having rank $r$, for $i=1,2$, and $\RR(U_1'') = \RR(U_2'')$.
Let $0\neq u_i \in \NN(U_i^{\prime\prime})$, for $i=1,2$.
If $\bar U \in \MM^{k \times (t+1)}$ satisfies
$\RR(\bar U) = \RR (\hat U_1) \cap  \RR (\hat U_2)$, then $t=r+1$ and
\beq
\label{eq:rangeUsdeg}
\begin{array}{rcl}
\RR(\bar U)
&=&  
\RR\left(\begin{bmatrix} 
  U_1^\prime& 0\cr
  U_1^{\prime\prime}& 0\cr
  U_2^{\prime} (U_2^{\prime\prime})^\dagger U_1^{\prime\prime} &
              U_2^\prime u_2
\end{bmatrix} \right)
=
\RR\left(
\begin{bmatrix} 
\bar U_1
&
\begin{bmatrix} 
 0\cr
 0\cr
              U_2^\prime u_2
\end{bmatrix} 
\end{bmatrix} 
\right)
\vspace{.1in}
\\&=&
\RR\left(\begin{bmatrix} 
  U_1^{\prime} (U_1^{\prime\prime})^\dagger U_2^{\prime\prime} & 
                  U_1^{\prime} u_1\cr
  U_2^{\prime\prime}& 0\cr
  U_2^\prime & 0
\end{bmatrix} \right)
=
\RR\left(
\begin{bmatrix} 
\bar U_2
&
\begin{bmatrix} 
              U_1^\prime u_1 \cr 
 0\cr
 0\cr
\end{bmatrix}  
\end{bmatrix}  
\right).
\end{array}
\eeq
Moreover, if $e_{r+1} \in \R^{r+1}$ is the $(r+1)^\mathrm{st}$ standard unit vector, and
$U_ie_{r+1} = \alpha_i e$, for some $\alpha_i \neq 0$, for $i=1,2$, then
$\bar U_ie_{r+1} = \alpha_i e$, for $i=1,2$.
\end{lem}
\begin{proof}
From the definitions, $x \in \RR(\bar U)$ if and only if 
\beq
\label{eq:xU12lemdeg}
  x = 
  \begin{bmatrix}
    x_1 \\
    x_2 \\
    x_3 
  \end{bmatrix}
  =
  \begin{bmatrix}
	  U^{\prime}_1  v_1 \\
    U^{\prime\prime}_1 v_1 \\
          v_2 
  \end{bmatrix}
  =
  \begin{bmatrix}
          w_1 \\
	  U^{\prime\prime}_2 w_2 \\
	  U^{\prime}_2  w_2 
  \end{bmatrix}, 
  \mbox{ for some }
v = \begin{bmatrix} v_1 \cr v_2 \end{bmatrix}, 
w = \begin{bmatrix} w_1 \cr w_2 \end{bmatrix}.
\eeq
Since $\RR(U_1'') = \RR(U_2'')$, 
and  $U''_i, i=1,2$, are both rank $r$, we conclude that
	  $x_2=U^{\prime\prime}_1 v_1 = U^{\prime\prime}_2 w_2$, 
for some $v_1,w_2$ if and only if $x_2 \in \RR( U^{\prime\prime}_1)$, with
$v_1,w_2$ determined by
\[
v_1=(U''_1)^\dagger x_2 + \alpha_1u_1,  \mbox{ for some } \alpha_1\in \R,
\qquad  w_2=(U''_2)^\dagger U''_1 v_1 + \alpha_2 u_2, \mbox{ for some }
\alpha_2 \in \R.
\]
In other words, we get 
\beq
\label{eq:iffxvwdeg}
\begin{array}{c}
	  x_2=U^{\prime\prime}_1 v_1 = U^{\prime\prime}_2 w_2, 
\mbox{ for some } v_1,w_2,  \\
\mbox{ if and only if } \\
x_2=U_1''v_1, \mbox{ for some } v_1, \mbox{ with }
  w_2=(U''_2)^\dagger U''_1 v_1+\alpha_2 u_2, \mbox{ for some } \alpha_2 \in \R.
\end{array}
\eeq
After substituting for $v_2$ with 
$v_2=U_2' w_2 = U_2' \left( (U_2'')^\dagger U_1''v_1 + 
\alpha_2 u_2\right)$, 
we conclude that \eqref{eq:xU12lemdeg} holds if and only if 
the first equality in \eqref{eq:rangeUsdeg} holds; i.e.,  \eqref{eq:xU12lemdeg} holds if and only if
\[
  x = 
  \begin{bmatrix}
    x_1 \\
    x_2 \\
    x_3 
  \end{bmatrix}
  =
  \begin{bmatrix}
	  U^{\prime}_1  v_1 \\
    U^{\prime\prime}_1 v_1 \\
  U_2^{\prime} (U_2^{\prime\prime})^\dagger U_1^{\prime\prime} v_1
                       +\alpha_2 U_2^{\prime} u_2
  \end{bmatrix},
  \mbox{ for some } v_1, \alpha_2,
\]
where
\[
v_2=U_2^\prime(U_2^{\prime\prime})^\dagger U_1^{\prime\prime} v_1 
+\alpha_2 U_2^\prime u_2, \quad
w_1=U_1'v_1, \quad
w_2=(U_2^{\prime\prime})^\dagger U_1^{\prime\prime} v_1 
+\alpha_2 u_2.
\]

The second equality in \eqref{eq:rangeUsdeg} follows similarly.
The last statements about $\bar U_ie_{r+1}$ follow as in the proof of
Lemma~\ref{lem:twocliquered}.
\end{proof}
In the rigid case in Theorem~\ref{thm:twocliquered}, 
we use the expression for $\bar U$
from Lemma~\ref{lem:twocliquered} to obtain a unique $Z$ in order to
get the completion of $D[\alpha_1 \cup \alpha_2]$. The $Z$ is unique because the
$r+1$ columns of $\bar U$ that represent the new clique
$\alpha_1\cup \alpha_2$ are linearly independent,
$e\in \RR(\bar U)$, $\rank(B)=r$, and $Be=0$. This means that the solution
$C$ of $(J \bar U_\beta \bar V)C = QD^{1/2}$ in Remark~\ref{rem:effZcalc}
exists and is unique. (Recall that 
$J \bar U_\beta \bar V$ is full column rank.)
This also means that the two
matrices, $U_1$ and $U_2$, that represent the cliques, $\alpha_1$ and $\alpha_2$,
respectively, can be replaced by the single matrix $\bar U$ without
actually calculating $C$; we can use $\bar U$  to represent
the clique $\alpha_1 \cup \alpha_2$ and complete all or part of the 
partial \edm $D[\alpha_1 \cup \alpha_2]$ only when needed.

We have a similar situation for the singular intersection case
following Lemma~\ref{lem:twocliquereddeg}. 
We have the matrix $\bar U$ to represent the intersection of the two subspaces,
where each subspace represents one of the cliques, $\alpha_1$ and $\alpha_2$.
However, this is not
equivalent to {\em uniquely} representing the union of the
two cliques, $\alpha_1$ and $\alpha_2$, 
since there is an extra column in $\bar U$ compared to the
nonsingular case. 
In addition, since $\rank(B)=r-1$, then $J \bar U_\beta \bar V$ is not 
necessarily full column rank.
Therefore, there may be infinite solutions for $C$ in
Remark~\ref{rem:effZcalc}; any
$C \in (J \bar U_\beta \bar V)^\dagger \left(QD^{1/2}\right) +
\NN(J \bar U_\beta \bar V)$ will give us a solution. 
Moreover, these solutions will not
necessarily satisfy  $\KK\left((\bar U C) (\bar UC)^T\right)  = D[\alpha_1 \cup \alpha_2]$.
We now see that we can continue and use the
$\bar U$ to represent a set of cliques rather than  just $\alpha_1 \cup \alpha_2$. 
Alternatively, we can use other relevant distance equality constraints or lower bound constraints from the radio range $R$ to determine the correct $C$ in order to get the
correct number of columns for $\bar U$; we can then get the correct
completion of $D[\alpha_1 \cup \alpha_2]$ if exactly one of the two possible completions with embedding dimension $r$ is feasible.

\begin{thm}
\label{thm:degcompl}
Let the hypotheses of Theorem~\ref{thm:twocliquered} hold with the
special case that 
$U_i^TU_i=I$, $U_ie_{r+1}=\alpha_ie$, for $i=1,2$.
In addition, let 
$\bar U$ be defined by one of the expressions in
\eqref{eq:rangeUsdeg} in Lemma~\ref{lem:twocliquereddeg}.
For $i=1,2$, let $\beta \subset \delta_i \subseteq \alpha_i$ and
$A_i := J \bar U_{\delta_i} \bar V$, where 
$\bar U_{\delta_i} := \bar U(\delta_i,:)$. Furthermore,
let $B_i := \KK^\dagger(D[\delta_i])$, define the linear system
\begin{equation}
  \label{eq:A1A2}
  \begin{array}{rcl}
    A_1 Z A_1^T & = & B_1 \\
    A_2 Z A_2^T & = & B_2,
  \end{array}
\end{equation}
and let $\bar Z \in \Ss^t$ be a 
particular solution of this system \eqref{eq:A1A2}.
If the embedding dimensions of $D[\delta_1]$ and $D[\delta_2]$ 
are both $r$, but the embedding dimension of $\bar D := D[\beta]$ is $r-1$, 
then the following holds.
\begin{enumerate}
  \item 
	  \label{item:cornulls}
    $\dim \NN(A_i) = 1$, for $i=1,2$.
  \item 
	  \label{item:corsolsZ}
For $i=1,2$, let $n_i \in \NN(A_i)$, $\|n_i\|_2 = 1$, 
    and $\Delta\!Z := n_1 n_2^T + n_2 n_1^T$. Then, 
    $Z$ is a solution of the linear system~\eqref{eq:A1A2} if and only if
    \beq
\label{eq:solnZ}
      Z = \bar Z + \tau \Delta\!Z, 
      \quad \mbox{for some $\tau \in \R$}.
    \eeq
  \item
	  \label{item:corcompletion}
    There are at most two nonzero solutions, $\tau_1$ and $\tau_2$,
          for the generalized eigenvalue
        problem $-\Delta\!Z v = \tau \bar Z v$, $v \neq 0$.
    Set $Z_i := \bar Z + \frac{1}{\tau_i} \Delta\!Z$,  for $i=1,2$. Then
    \[
      D[\alpha_1 \cup \alpha_2] \in \left\{
        \KK(\bar U \bar V Z_i \bar V^T \bar U^T) : i=1,2 \right\}.
    \]
\end{enumerate}
\end{thm}

\begin{proof}
We follow a similar proof as in the nonsingular case.
For simplicity, we assume that $\delta_i=\alpha_i$, for $i=1,2$
(choosing smaller $\delta_i$ can reduce the cost of solving the linear
systems).

That a particular solution $\bar Z$ 
exists for the system \eqref{eq:A1A2}, follows
from the fact that $\bar U$ provides a representation for the
intersection of the two faces (or the union of the two cliques).

Since the embedding dimension of $\bar D$ is $r-1$, we have $\rank(B) = r-1$.  
Furthermore, we have $Be = 0$ and $B \in \Ss^{|\beta|}_+$, implying that
$|\beta| \geq r$.  Without loss of generality, and for simplicity, we assume
that $|\beta| = r$. Therefore, there exists $0\neq u_i \in
\NN(U_i^{\prime\prime})$, for $i=1,2$. From Lemma~\ref{lem:twocliquereddeg},
we can assume that we maintain $\bar U_i^T \bar U_i=I$, $\bar U_ie_{r+1}=\alpha_i
e$, for some $\alpha_i \neq 0$, for $i=1,2$. Therefore, the action of $\bar V$ is equivalent
to removing the $r+1$ column of $\bar U_i$. We can then explicitly use
$u_i$ to write down $n_i \in \NN(A_i)$. 
By construction, we now have $A_i(n_1 n_2^T + n_2 n_1^T)A_i^T=0$, for $i=1,2$.

From the first expression for $\bar U$
in \eqref{eq:rangeUsdeg}, we see that the choices for $n_1$ and $n_2$ in 
Part~\ref{item:cornulls} are in the appropriate nullspaces. The
dimensions follow from the assumptions on the embedding dimensions.

Part~\ref{item:corsolsZ} now follows from the definition of the general
solution of a linear system of equations; i.e., the sum of a particular
solution with any solution of the homogeneous equation.

Part~\ref{item:corcompletion} now follows from the role that $\bar U$
plays as a representation for the union of the two cliques.
\end{proof}
\begin{remark}
\label{rem:effZcalcdeg}
As above in the nonsingular case,
a more efficient way to calculate $\bar Z$ uses the full rank factorization
\[
B_i=QD^{1/2}\left(Q_iD_i^{1/2}\right)^T, \quad
Q_i^TQ_i=I_r, \quad
D_i\in \Ss_{++}^{r}, \quad 
i=1,2.
\]
(We have assumed that both have embedding dimension $r$, though we
only need that one does.)
We solve the equations $A_i C = \left(Q_iD_i^{1/2}\right)\bar Q_i$, 
$\bar Q_i\bar Q_i^T=I$, for $i=1,2$,
for the unknowns $C$, and $\bar Q_i$, for $i=1,2$.
Then a particular solution
$\bar Z$ in \eqref{eq:A1A2} can be found from $\bar Z=CC^T$.
Note that the additional orthogonal matrices $\bar Q_i$, for $i=1,2$ are needed
since, they still allow $A_iC(A_iC)^T=B_i$, for $i=1,2$. 
Also, without loss of generality, we can
assume $\bar Q_1=I$.
  
\end{remark}

\subsection{Clique Initialization and Node Absorption}

Using the above clique reductions,
we now consider techniques that allow one clique to grow/absorb other
cliques.  This applies Theorem~\ref{thm:twocliquered}.
We first consider an elementary and fast technique to find some of the
existing cliques.

\begin{lem}
	\label{lem:gencliques}
For each $i \in \{1,\ldots,n\}$, use half the radio range
and define the set
\[
C_i := \left\{ j \in \{1,\ldots,n\} : D_{ij} \leq (R/2)^2 \right\}.
\]
Then each $C_i$ corresponds to a clique of sensors that are
within radio range of each other.
\index{half radio range clique centered at node $i$, $C_i$}
\end{lem}

\begin{proof}
	Let $j,k \in C_i$ for a given $i \in \{1,\ldots,n\}$.
An elementary application of the triangle inequality shows that
$\sqrt{(D_{jk})} \leq \sqrt{(D_{ji})} + \sqrt{(D_{ki})} \leq R$.
\end{proof}

We can now assume that we have a finite set of indices
$\CC \subseteq \mathbb{Z}_{+}$  corresponding to a family of cliques,
\index{clique index set, $\CC$}
$\{C_i\}_{i\in \CC}$. We can combine cliques using the reductions given in 
Theorems~\ref{thm:twocliquered} and \ref{thm:degcompl}.
We now see how a clique can grow further by absorbing individual sensors;
see Figure \ref{fig:cliqabsorb}.

\begin{figure}[htb]
\epsfxsize=160pt
\centerline{\epsfbox{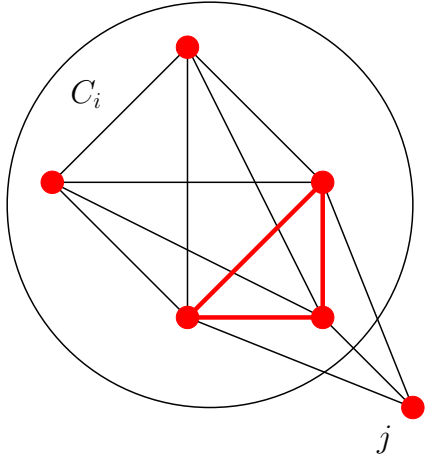}}
\caption{Absorption with intersection having embedding 
       dimension $r$}
\label{fig:cliqabsorb}
\end{figure}

\begin{cor}
	\label{cor:absorbsingle}
Let  $C_k$, for $k\in \CC$, be a given clique with node $l \notin  C_k$,
$\beta := \left\{j_1,\ldots,j_{r+1}\right\} \subseteq C_k$, 
such that the
distances $D_{lj_i}$, for $i=1,\ldots,r+1$ are known.
If
\beq
\label{eq:rankabsorb}
\rank \KK^\dagger (D[\beta]) = r,
\eeq
then $l$ can be absorbed by the clique $C_k$ and we can complete
the missing elements in column (row) $l$ of $D[C_k\cup \{l\}]$.
\end{cor}
\begin{proof}
	Let $\alpha_1:=C_k$, $\alpha_2:=\{j_1,\ldots,j_{r+1},l\}$, and 
	$\beta := \alpha_1 \cap \alpha_2 = \{j_1,\ldots,j_{r+1}\}$. Then the
conditions in Theorem~\ref{thm:twocliquered} are satisfied and we
can recover all the missing elements in $D[C_k\cup \{l\}]$.
\end{proof}

\subsubsection{Node Absorption with Degenerate Intersection}

We can apply the same reasoning as for the clique reduction in the
nonsingular case, except now we apply Theorem~\ref{thm:degcompl}. 
To obtain a unique completion, we test the feasibility of the two possible completions against any related distance equality constraints or, if included, any related lower bound inequality constraints.
See Figure~\ref{fig:degnodeabsorg}.

\begin{figure}[htb]
\epsfxsize=160pt
\centerline{\epsfbox{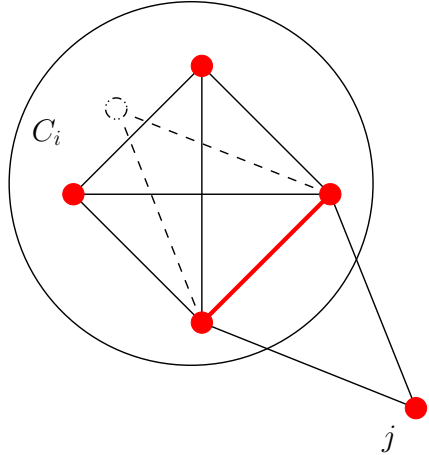}}
\caption{Degenerate absorption with intersection with embedding 
       dimension $<r$}
\label{fig:degnodeabsorg}
\end{figure}

\begin{cor}
\label{cor:degabsorbsingle}
Let  $C_k$, for $k\in \CC$, be a given clique with node $l \notin  C_k$,
$\beta := \left\{j_1,\ldots j_{r}\right\} \subseteq C_k$
such that the distances $D_{lj_i}$, for $i=1,\ldots,r$ are known. If
\beq
\label{eq:rankabsorbdeg}
\rank \KK^\dagger (D[\beta]) = r-1,
\eeq
then we can determine two possible completions of the distances.  If exactly one of these two completions is feasible, then $l$ can be absorbed by the clique $C_k$. We can also complete
the missing elements in column (row) $l$ of $D[C_k\cup \{l\}]$.
\end{cor}

\begin{proof}
	Let $\alpha_1:=C_k$, $\alpha_2:=\{j_1,\ldots,j_{r},l\}$, and 
	$\beta := \alpha_1 \cap \alpha_2 = \{j_1,\ldots ,j_{r}\}$. Then the
conditions in Theorem~\ref{thm:degcompl} are satisfied and we
can recover all the missing elements in $D[C_k\cup \{l\}]$.
\end{proof}

\section{\texttt{SNLSDPclique} Facial Reduction Algorithm and Numerical Results}
\label{sect:algor}

Our \texttt{SNLSDPclique} algorithm starts by forming a clique $C_i$ around each
sensor $i$.
If and when we use this clique, we find a subspace representation
from the $r$ eigenvectors corresponding to the $r$ nonzero
eigenvalues of $B=\KK^\dagger(D[C_i])$.

The algorithm then grows and combines cliques using 
Theorem~\ref{thm:twocliquered}, Theorem~\ref{thm:degcompl},
Corollary~\ref{cor:absorbsingle}, and Corollary~\ref{cor:degabsorbsingle}.
In particular, we do not complete the \EDM each time we combine or grow
cliques; i.e., we do not evaluate the missing distances. Instead, we use 
the subspace representations of the corresponding faces of the \SDP cone 
and then find the intersection of the subspaces that represent the faces. 
This yields a subspace representation of the new smaller 
face representing the union of two cliques. This is based on 
Lemma~\ref{lem:twocliquered} and Lemma~\ref{lem:twocliquereddeg} and is therefore
inexpensive.

Once we cannot, or need not, grow cliques, we complete the distances
using Corollary~\ref{cor:finalUbar}. This is also inexpensive.
Finally, we rotate and translate the anchors to their original positions using
the approach outlined in \cite{DiKrQiWo:06}.  We have provided an outline of our facial reduction algorithm \texttt{SNLSDPclique} in Algorithm~\ref{alg:SNLSDPclique}.

\begin{algorithm}
\caption{\texttt{SNLSDPclique} -- a facial reduction algorithm}
\label{alg:SNLSDPclique}

\SetAlgoLined
\LinesNumbered

\SetKwFunction{Face}{Face}
\SetKwFunction{RigidCliqueUnion}{RigidCliqueUnion}
\SetKwFunction{RigidNodeAbsorption}{RigidNodeAbsorption}
\SetKwFunction{NonRigidCliqueUnion}{NonRigidCliqueUnion}
\SetKwFunction{NonRigidNodeAbsorption}{NonRigidNodeAbsorption}

\SetKwInOut{Input}{input}
\SetKwInOut{Output}{output}

\Input{Partial $n \times n$ Euclidean Distance Matrix $D_p$ and anchors $A \in \R^{m \times r}$\;}
\Output{$X \in \R^{|C_i| \times r}$, where $C_i$ is the largest final clique that contains the anchors\;}

\BlankLine
Let $\mathcal{C} := \{1,\ldots,n+1\}$\;
Let $\{C_i\}_{i \in \mathcal{C}}$ be a family of cliques satisfying $i \in C_i$ for all $i = 1,\ldots,n$%
\tcc*{For example, by Lemma~\ref{lem:gencliques}, we could choose $C_i := \left\{ j : (D_p)_{ij} < (R/2)^2 \right\}$, for $i=1,\ldots,n.$  Alternatively, we could simply choose $C_i := \{i\}$, for $i=1,\ldots,n$.}
Let $C_{n+1} := \{n-m+1,\ldots,n\}$%
\tcc*{$C_{n+1}$ is the clique of anchors}

\BlankLine
\tcc{GrowCliques}
Choose $\mbox{\sc{MaxCliqueSize}} > r+1$%
\tcc*{For example, $\mbox{\sc{MaxCliqueSize}} := 3(r+1)$}
\For{$i \in \mathcal{C}$}{
	\While{($|C_i| < \mbox{\sc{MaxCliqueSize}}$)
			{\bf and} ($\exists$ a node $j$ adjacent to all nodes in $C_i$)}{
		$C_i := C_i \cup \{j\}$\;
	}
}

\BlankLine
\tcc{ComputeFaces}
\For{$i \in \mathcal{C}$}{
	Compute $U_{B_i} \in \R^{|C_i| \times (r+1)}$ to represent face for clique $C_i$%
	\tcc*{see Theorem~\ref{thm:onecliquered}}
	\tcc{Alternatively, wait to compute $U_{B_i}$ when first needed.  This can be more efficient since $U_{B_i}$ is not needed for every clique.}
}

\BlankLine
\Repeat{not possible to  decrease $|\mathcal{C}|$ or increase $|C_i|$ for some $i \in \mathcal{C}$}{
	\uIf{$|C_i \cap C_j| \geq r+1$, for some $i,j \in \mathcal{C}$}{
		\RigidCliqueUnion{$C_i$,$C_j$}%
		\tcc*{see Algorithm~\ref{alg:RigidCliqueUnion}}
	}
	\uElseIf{$|C_i \cap \NN(j)| \geq r+1$, for some $i \in \mathcal{C}$ and node $j$}{
		\RigidNodeAbsorption{$C_i$,$j$}%
		\tcc*{see Algorithm~\ref{alg:RigidNodeAbsorption}}
	} 
	\uElseIf{$|C_i \cap C_j|= r$, for some $i,j \in \mathcal{C}$}{
		\NonRigidCliqueUnion{$C_i$,$C_j$}%
		\tcc*{see Algorithm~\ref{alg:NonRigidCliqueUnion}}
	}
	\ElseIf{$|C_i \cap \NN(j)| = r$, for some $i \in \mathcal{C}$ and node $j$}{
		\NonRigidNodeAbsorption{$C_i$,$j$}%
		\tcc*{see Algorithm~\ref{alg:NonRigidNodeAbsorption}}
	}
}

\BlankLine
Let $C_i$ be the largest clique that contains the anchors\;
\uIf{clique $C_i$ contains some sensors}{
	Compute a point representation $P \in \R^{|C_i| \times r}$ for the clique $C_i$%
	\tcc*{see Cor.~\ref{cor:finalUbar}}
	Compute positions of sensors $X \in \R^{(|C_i|-m)\times r}$ in clique $C_i$ by rotating $P$ to align with anchor positions $A \in \R^{m \times r}$%
	\tcc*{see Ding et al.~\cite[Method~3.2]{DiKrQiWo:06}}
	\Return $X$\;
}
\Else{
	\Return $X := \emptyset$\;
}
\end{algorithm}

\begin{algorithm}
\caption{\texttt{RigidCliqueUnion}}
\label{alg:RigidCliqueUnion}
\SetAlgoLined
\LinesNumbered
\SetKwInOut{Input}{input}
\SetKwInOut{Output}{output}
\SetKwInOut{Empty}{}
\Input{Cliques $C_i$ and $C_j$ such that $|C_i \cap C_j| \geq r+1$\;}
\BlankLine
Load $U_{B_i} \in \R^{|C_i|\times(r+1)}$ and $U_{B_j} \in \R^{|C_j|\times(r+1)}$ representing the faces corresponding to the cliques $C_i$ and $C_j$, respectively\;
Compute $\bar U \in \R^{|C_i \cup C_j| \times (r+1)}$ using one of the two formulas in equation~\eqref{eq:U1U2} from Lemma~\ref{lem:twocliquered}, where $U_1 = U_{B_i}$, $U_2 = U_{B_j}$, and  $k = |C_i \cap C_j|$%
\tcc*{see Theorem~\ref{thm:interclique}}
Update $C_i := C_i \cup C_j$\;
Update $U_{B_i} := \bar U$\; 
Update $\mathcal{C} := \mathcal{C} \setminus\{j\}$\;
\end{algorithm}

\begin{algorithm}
\caption{\texttt{RigidNodeAbsorption}}
\label{alg:RigidNodeAbsorption}
\SetAlgoLined
\LinesNumbered
\SetKwInOut{Input}{input}
\SetKwInOut{Output}{output}
\SetKwInOut{Empty}{}
\Input{Clique $C_i$ and node $j$ such that $|C_i \cap \NN(j)| \geq r+1$\;} 
\BlankLine
Load $U_{B_i} \in \R^{|C_i|\times(r+1)}$ representing the face corresponding to clique $C_i$\;
\If{$C_i \cap \NN(j)$ not a clique in the original graph}{
	Use $U_{B_i}$ to compute a point representation $P_i \in \R^{|C_i| \times r}$ of the sensors in $C_i$\;
	\tcc*[f]{see Cor.~\ref{cor:finalUbar}} \\
	Use $P_i$ to compute the distances between the sensors in $C_i \cap \NN(j)$\;
}
Use the distances between the sensors in $(C_i \cap \NN(j))\cup\{j\}$ to compute the matrix $U_{B_j} \in \R^{(|C_i \cap \NN(j)|+1)\times(r+1)}$ representing the face corresponding to the clique $(C_i \cap \NN(j))\cup\{j\}$%
\tcc*{see Theorem~\ref{thm:onecliquered}}
Compute $\bar U \in \R^{(|C_i|+1) \times (r+1)}$  using one of the two formulas in equation~\eqref{eq:U1U2} from Lemma~\ref{lem:twocliquered}, where $U_1 = U_{B_i}$, $U_2 = U_{B_j}$, and $k = |C_i \cap \NN(j)|$%
\tcc*{see Theorem~\ref{thm:interclique}}
Update $C_i := C_i \cup \{j\}$\;
Update $U_{B_i} := \bar U$\;
\end{algorithm}

\begin{algorithm}
\caption{\texttt{NonRigidCliqueUnion}}
\label{alg:NonRigidCliqueUnion}
\SetAlgoLined
\LinesNumbered
\SetKwInOut{Input}{input}
\SetKwInOut{Output}{output}
\SetKwInOut{Empty}{}
\Input{Cliques $C_i$ and $C_j$ such that $|C_i \cap C_j| = r$\;}
\BlankLine
Load $U_{B_i} \in \R^{|C_i|\times(r+1)}$ and $U_{B_j} \in \R^{|C_j|\times(r+1)}$ representing the faces corresponding to the cliques $C_i$ and $C_j$, respectively\;
Using $U_{B_i}$ and $U_{B_j}$, find the two point representations of the sensors in $C_i \cup C_j$\; 
\tcc*[f]{see Theorem~\ref{thm:degcompl}} \\
\If{exactly one of these two point representations is feasible}{
	Use the feasible point representation to compute $\bar U \in \R^{|C_i \cup C_j| \times (r+1)}$ representing the face corresponding to the clique $C_i \cup C_j$%
	\tcc*{see Theorem~\ref{thm:onecliquered}}
	Update $C_i := C_i \cup C_j$\;
	Update $U_{B_i} := \bar U$\; 
	Update $\mathcal{C} := \mathcal{C} \setminus\{j\}$\;
}
\end{algorithm}

\begin{algorithm}
\caption{\texttt{NonRigidNodeAbsorption}}
\label{alg:NonRigidNodeAbsorption}
\SetAlgoLined
\LinesNumbered
\SetKwInOut{Input}{input}
\SetKwInOut{Output}{output}
\SetKwInOut{Empty}{}
\Input{Clique $C_i$ and node $j$ such that $|C_i \cap \NN(j)| = r$\;} 
\BlankLine
Load $U_{B_i} \in \R^{|C_i|\times(r+1)}$ representing the face corresponding to clique $C_i$\;
\If{$C_i \cap \NN(j)$ not a clique in the original graph}{
	Use $U_{B_i}$ to compute a point representation $P_i \in \R^{|C_i| \times r}$ of the sensors in $C_i$\;
	\tcc*[f]{see Cor.~\ref{cor:finalUbar}} \\
	Use $P_i$ to compute the distances between the sensors in $C_i \cap \NN(j)$\;
}
Use the distances between the sensors in $(C_i \cap \NN(j))\cup\{j\}$ to compute the matrix $U_{B_j} \in \R^{(|C_i \cap \NN(j)|+1)\times(r+1)}$ representing the face corresponding to the clique $(C_i \cap \NN(j))\cup\{j\}$%
\tcc*{see Theorem~\ref{thm:onecliquered}}

Using $U_{B_i}$ and $U_{B_j}$, find the two point representations of the sensors in $C_i \cup \{j\}$\; 
\tcc*[f]{see Theorem~\ref{thm:degcompl}} \\
\If{exactly one of these two point representations is feasible}{
	Use the feasible point representation to compute $\bar U \in \R^{|C_i \cup C_j| \times (r+1)}$ representing the face corresponding to the clique $C_i \cup \{j\}$%
	\tcc*{see Theorem~\ref{thm:onecliquered}}
	Update $C_i := C_i \cup \{j\}$\;
	Update $U_{B_i} := \bar U$\; 
}
\end{algorithm}

\subsection{Numerical Tests}

\label{sect:numerics}
Our tests are on problems with sensors and anchors randomly placed in the region $[0,1]^r$ by means of a uniform random distribution.  
We vary the number of sensors from $2000$ to $10000$ in steps of
$2000$, and the radio range $R$ from $.07$ to $.04$ in steps of $-.01$.
We also include tests on very large problems with $20000$ to $100000$ sensors.
In our tests, we did not use the lower bound inequality constraints coming from the radio range; we only used the equality constraints coming from the partial Euclidean distance matrix.
Our tests were done using the 32-bit version of {\sc Matlab} R2009b on a laptop running Windows XP, with a 2.16 GHz Intel Core 2 Duo processor and with 2 GB of RAM.  The source code used for running our tests has been released under a GNU General Public License, and has been made available from the authors' websites.

We in particular
emphasize the low CPU times and the high accuracy of the solutions we obtain.
Our algorithm compares well with the recent work in 
\cite{pongtseng:09,WangZhengBoydYe:06}, where they use, for example,
 $R = .06$ for $n = 1000, 2000$,
 $R = .035$ for $n = 4000$,
 $R = .02$ for $n = 10000$, and also
use $10$\% of the sensors as anchors and limit the degree for each
node in order to maintain a low sparsity for the graph.

Tables~\ref{table:RigidCliqueUnion}, \ref{table:RigidNodeAbsorb}, and \ref{table:NonRigidCliqueUnion} contain the results of our tests on noiseless problems.  These tables contain the following information.
\begin{enumerate}

\item
{\bf \# sensors, $r$, \# anchors, and $R$:}
We use $m = (\# anchors)$, $n = (\# sensors) + (\# anchors)$, and $r$ to generate ten random instances of $p_1,\ldots,p_n \in \R^r$; the last $m$ points are taken to be the anchors.  For each of these ten instances, and for each value of the radio range $R>0$, we generate the the $n \times n$ partial Euclidean distance matrix $D_p$ according to
\[
	(D_p)_{ij} = 
	\begin{cases}
		\|p_i-p_j\|^2, & \text{if $\|p_i-p_j\|<R$, or both $p_i$ and $p_j$ are anchors} \\
		\text{unspecified}, & \text{otherwise}.
	\end{cases}
\]

\item
{\bf \# Successful Instances:}
An instance was called \emph{successful} if at least some, if not all, of the sensors could be positioned.  If, by the end of the algorithm, the largest clique containing the anchors did not contain any sensors, then none of the sensor positions could be determined, making such an instance unsuccessful.

\item
{\bf Average Degree:}
We have found that the average degree of the nodes of a graph is a good indicator of the percentage of sensors that can be positioned.  In the results reported, we give the average of the average degree over all ten instances.

\item
{\bf \# Sensors Positioned:}
We give the average number of sensors that could be positioned over all ten instances.   Note that below we indicate that the error measurements are computed only over the sensors that could be positioned.

\item
{\bf CPU Time:}
Indicates the average running time of {\texttt{SNLSDPclique}} over all ten instances.  This time does not include the time to generate the random problems, but it does include all aspects of the Algorithm~\ref{alg:SNLSDPclique}, including the time for {\texttt{GrowCliques}} and {\texttt{ComputeFaces}} at the beginning of the algorithm.

\item
{\bf Max Error:}  
The maximum distance between the positions of the sensors
found and the true positions of those sensors. This is defined as
\[
	\mbox{Max Error} := \max_{\mbox{\tiny $i$ positioned} } \| p_i - p_i^{\mathrm true} \|_2.
\]

\item
{\bf RMSD:}  The root-mean-square deviation of the positions of the sensors
found and the true positions of those sensors. This is defined as
\[
	\mbox{RMSD} := 
	\left(
	\frac{1}{\mbox{\# positioned}}
	\sum_{\mbox{\tiny $i$ positioned} } \| p_i - p_i^{\mathrm true} \|_2^2
	\right)^\frac12.
\]

\end{enumerate}
We note that for each set of ten random instances, the Max Error and RMSD values reported are the average Max Error and average RMSD values over the successful instances only; this is due to the fact that an unsuccessful instance will have no computed sensor positions to compare with the true sensor positions.

We have three sets of tests on noiseless problems.
\begin{enumerate}
\item 
In Table~\ref{table:RigidCliqueUnion} we report the results of using only the {\texttt{RigidCliqueUnion}} step (see Figure~\ref{fig:intersembedr}) to solve our random problems.

\begin{table}
\caption[\texttt{RigidCliqueUnion}]{Results of Algorithm~\ref{alg:SNLSDPclique} 
on noiseless problems,
using step {\texttt{RigidCliqueUnion}}.
The values for Average Degree, \# Sensors Positioned, and CPU Time are averaged over ten random instances.  The values for Max~Error and RMSD values are averaged over the successful instances.}
\label{table:RigidCliqueUnion}
\begin{center}
\begin{footnotesize}
\begin{tabular}{|c|c|c|c||c|c|c|c|c|c|}
\hline
           &     &            &     & \# Successful & Average & \# Sensors &          &           &      \\           
\# sensors & $r$ & \# anchors & $R$ & Instances     &  Degree & Positioned & CPU Time & Max Error & RMSD \\
\hline
2000 & 2 & 4 & .07 &  9/10 & 14.5 & 1632.3 & 1 s & 6e-13 & 2e-13 \\
2000 & 2 & 4 & .06 &  5/10 & 10.7 &  720.0 & 1 s & 1e-12 & 4e-13 \\
2000 & 2 & 4 & .05 &  0/10 &  7.5 &    0.0 & 1 s &   -   &   -   \\
2000 & 2 & 4 & .04 &  0/10 &  4.9 &    0.0 & 1 s &   -   &   -   \\
\hline
4000 & 2 & 4 & .07 & 10/10 & 29.0 & 3904.1 & 2 s & 2e-13 & 6e-14 \\
4000 & 2 & 4 & .06 & 10/10 & 21.5 & 3922.3 & 2 s & 6e-13 & 2e-13 \\
4000 & 2 & 4 & .05 & 10/10 & 15.1 & 3836.2 & 2 s & 4e-13 & 2e-13 \\
4000 & 2 & 4 & .04 &  1/10 &  9.7 &  237.8 & 2 s & 1e-13 & 4e-14 \\
\hline
6000 & 2 & 4 & .07 & 10/10 & 43.5 & 5966.9 & 4 s & 3e-13 & 8e-14 \\
6000 & 2 & 4 & .06 & 10/10 & 32.3 & 5964.4 & 4 s & 2e-13 & 7e-14 \\
6000 & 2 & 4 & .05 & 10/10 & 22.6 & 5894.8 & 3 s & 3e-13 & 1e-13 \\
6000 & 2 & 4 & .04 & 10/10 & 14.6 & 5776.9 & 3 s & 7e-13 & 2e-13 \\
\hline
8000 & 2 & 4 & .07 & 10/10 & 58.1 & 7969.8 & 6 s & 3e-13 & 8e-14 \\
8000 & 2 & 4 & .06 & 10/10 & 43.0 & 7980.9 & 6 s & 2e-13 & 8e-14 \\
8000 & 2 & 4 & .05 & 10/10 & 30.1 & 7953.1 & 5 s & 6e-13 & 2e-13 \\
8000 & 2 & 4 & .04 & 10/10 & 19.5 & 7891.0 & 5 s & 6e-13 & 2e-13 \\
\hline
10000 & 2 & 4 & .07 & 10/10 & 72.6 & 9974.6 & 9 s & 3e-13 & 7e-14 \\
10000 & 2 & 4 & .06 & 10/10 & 53.8 & 9969.1 & 8 s & 9e-13 & 1e-13 \\
10000 & 2 & 4 & .05 & 10/10 & 37.7 & 9925.4 & 7 s & 5e-13 & 2e-13 \\
10000 & 2 & 4 & .04 & 10/10 & 24.3 & 9907.2 & 7 s & 3e-13 & 1e-13 \\
\hline
\hline
 20000 & 2 & 4 & .030 & 10/10 & 27.6 & 19853.3 &     17 s & 7e-13 & 2e-13 \\
 40000 & 2 & 4 & .020 & 10/10 & 24.7 & 39725.2 &     50 s & 2e-12 & 6e-13 \\
 60000 & 2 & 4 & .015 & 10/10 & 21.0 & 59461.1 & 1 m 52 s & 1e-11 & 8e-13 \\
 80000 & 2 & 4 & .013 & 10/10 & 21.0 & 79314.1 & 3 m 24 s & 4e-12 & 1e-12 \\
100000 & 2 & 4 & .011 & 10/10 & 18.8 & 99174.4 & 5 m 42 s & 2e-10 & 9e-11 \\
\hline
\end{tabular}
\end{footnotesize}
\end{center}
\end{table}

\item
In Table~\ref{table:RigidNodeAbsorb} we report the results of increasing the level of our algorithm to use both the {\texttt{RigidCliqueUnion} and \texttt{RigidNodeAbsorb}} steps (see Figures~\ref{fig:intersembedr} and \ref{fig:cliqabsorb}) to solve the random problems.
We see that
the number of sensors localized has increased and
that there has been a small, almost insignificant, increase in the CPU time.

\begin{table}
\caption[{\texttt{RigidCliqueUnion}} and {\texttt{RigidNodeAbsorb}}]{Results of Algorithm~\ref{alg:SNLSDPclique} 
on noiseless problems,  
using steps {\texttt{RigidCliqueUnion}} and {\texttt{RigidNodeAbsorb}}.
The values for Average Degree, \# Sensors Positioned, and CPU Time are averaged over ten random instances.  The values for Max~Error and RMSD values are averaged over the successful instances.}
\label{table:RigidNodeAbsorb}
\begin{center}
\begin{footnotesize}
\begin{tabular}{|c|c|c|c||c|c|c|c|c|c|}
\hline
           &     &            &     & \# Successful & Average & \# Sensors &          &           &      \\           
\# sensors & $r$ & \# anchors & $R$ & Instances     &  Degree & Positioned & CPU Time & Max Error & RMSD \\
\hline
2000 & 2 & 4 & .07 &  10/10 & 14.5 & 2000.0 & 1 s & 6e-13 & 2e-13 \\
2000 & 2 & 4 & .06 &  10/10 & 10.7 & 1999.9 & 1 s & 8e-13 & 3e-13 \\
2000 & 2 & 4 & .05 &  10/10 &  7.5 & 1996.7 & 1 s & 9e-13 & 2e-13 \\
2000 & 2 & 4 & .04 &   9/10 &  4.9 & 1273.8 & 3 s & 2e-11 & 4e-12 \\
\hline
4000 & 2 & 4 & .07 & 10/10 & 29.0 & 4000.0 & 2 s & 2e-13 & 6e-14 \\
4000 & 2 & 4 & .06 & 10/10 & 21.5 & 4000.0 & 2 s & 6e-13 & 2e-13 \\
4000 & 2 & 4 & .05 & 10/10 & 15.1 & 3999.9 & 2 s & 6e-13 & 3e-13 \\
4000 & 2 & 4 & .04 & 10/10 &  9.7 & 3998.2 & 2 s & 1e-12 & 5e-13 \\
\hline
6000 & 2 & 4 & .07 & 10/10 & 43.5 & 6000.0 & 4 s & 3e-13 & 8e-14 \\
6000 & 2 & 4 & .06 & 10/10 & 32.3 & 6000.0 & 4 s & 2e-13 & 7e-14 \\
6000 & 2 & 4 & .05 & 10/10 & 22.6 & 6000.0 & 3 s & 3e-13 & 1e-13 \\
6000 & 2 & 4 & .04 & 10/10 & 14.6 & 5999.4 & 3 s & 8e-13 & 3e-13 \\
\hline
8000 & 2 & 4 & .07 & 10/10 & 58.1 & 8000.0 & 6 s & 3e-13 & 7e-14 \\
8000 & 2 & 4 & .06 & 10/10 & 43.0 & 8000.0 & 5 s & 2e-13 & 8e-14 \\
8000 & 2 & 4 & .05 & 10/10 & 30.1 & 8000.0 & 5 s & 6e-13 & 2e-13 \\
8000 & 2 & 4 & .04 & 10/10 & 19.5 & 8000.0 & 4 s & 7e-13 & 2e-13 \\
\hline
10000 & 2 & 4 & .07 & 10/10 & 72.6 & 10000.0 & 9 s & 3e-13 & 7e-14 \\
10000 & 2 & 4 & .06 & 10/10 & 53.8 & 10000.0 & 8 s & 3e-13 & 1e-13 \\
10000 & 2 & 4 & .05 & 10/10 & 37.7 & 10000.0 & 7 s & 5e-13 & 2e-13 \\
10000 & 2 & 4 & .04 & 10/10 & 24.3 & 10000.0 & 6 s & 3e-13 & 1e-13 \\
\hline
\hline
 20000 & 2 & 4 & .030 & 10/10 & 27.6 &  20000.0 &     17 s & 7e-13 & 2e-13 \\
 40000 & 2 & 4 & .020 & 10/10 & 24.7 &  40000.0 &     51 s & 2e-12 & 6e-13 \\
 60000 & 2 & 4 & .015 & 10/10 & 21.0 &  60000.0 & 1 m 53 s & 2e-12 & 7e-13 \\
 80000 & 2 & 4 & .013 & 10/10 & 21.0 &  80000.0 & 3 m 21 s & 4e-12 & 1e-12 \\
100000 & 2 & 4 & .011 & 10/10 & 18.8 & 100000.0 & 5 m 46 s & 2e-10 & 9e-11 \\
\hline
\end{tabular}
\end{footnotesize}
\end{center}
\end{table}

\item
In Table~\ref{table:NonRigidCliqueUnion} we report the results of increasing the level of our algorithm to use steps \texttt{RigidCliqueUnion}, \texttt{RigidNodeAbsorb}, and \texttt{NonRigidCliqueUnion} (see Figures~\ref{fig:intersembedr}, \ref{fig:cliqabsorb}, and \ref{fig:deg2cliqred}) to solve the random problems, further increasing the class of problems that we can complete.

\end{enumerate}
Testing a version of our algorithm that uses all four steps is still ongoing.  From the above results, we can see that our facial reduction technique works very well for solving many instances of the \SNL problem.  We are confident that the results of our ongoing tests will continue to show that we are able to solve an even larger class of \SNL problems.

\begin{table}
\caption[{\texttt{RigidCliqueUnion}}, {\texttt{RigidNodeAbsorb}}, and {\texttt{NonRigidCliqueUnion}}]{Results of Algorithm~\ref{alg:SNLSDPclique} 
on noiseless problems,
using steps {\texttt{RigidCliqueUnion}}, {\texttt{RigidNodeAbsorb}}, and {\texttt{NonRigidCliqueUnion}}.
The values for Average Degree, \# Sensors Positioned, and CPU Time are averaged over ten random instances.  The values for Max~Error and RMSD values are averaged over the successful instances.
The results of the tests with more than $6000$ sensors remain the same as in Table~\ref{table:RigidNodeAbsorb}.}
\label{table:NonRigidCliqueUnion}
\begin{center}
\begin{footnotesize}
\begin{tabular}{|c|c|c|c||c|c|c|c|c|c|}
\hline
           &     &            &     & \# Successful & Average & \# Sensors &          &           &      \\           
\# sensors & $r$ & \# anchors & $R$ & Instances     &  Degree & Positioned & CPU Time & Max Error & RMSD \\
\hline
2000 & 2 & 4 & .07 &  10/10 & 14.5 & 2000.0 & 1 s & 6e-13 & 2e-13 \\
2000 & 2 & 4 & .06 &  10/10 & 10.7 & 1999.9 & 1 s & 8e-13 & 3e-13 \\
2000 & 2 & 4 & .05 &  10/10 &  7.5 & 1997.9 & 1 s & 9e-13 & 2e-13 \\
2000 & 2 & 4 & .04 &  10/10 &  4.9 & 1590.8 & 5 s & 2e-11 & 7e-12 \\
\hline
4000 & 2 & 4 & .07 & 10/10 & 29.0 & 4000.0 & 2 s & 2e-13 & 6e-14 \\
4000 & 2 & 4 & .06 & 10/10 & 21.5 & 4000.0 & 2 s & 6e-13 & 2e-13 \\
4000 & 2 & 4 & .05 & 10/10 & 15.1 & 3999.9 & 2 s & 6e-13 & 3e-13 \\
4000 & 2 & 4 & .04 & 10/10 &  9.7 & 3998.2 & 3 s & 1e-12 & 5e-13 \\
\hline
6000 & 2 & 4 & .07 & 10/10 & 43.5 & 6000.0 & 4 s & 3e-13 & 8e-14 \\
6000 & 2 & 4 & .06 & 10/10 & 32.3 & 6000.0 & 4 s & 2e-13 & 7e-14 \\
6000 & 2 & 4 & .05 & 10/10 & 22.6 & 6000.0 & 3 s & 3e-13 & 1e-13 \\
6000 & 2 & 4 & .04 & 10/10 & 14.6 & 5999.4 & 3 s & 8e-13 & 3e-13 \\
\hline
\end{tabular}
\end{footnotesize}
\end{center}
\end{table}

\subsection{Noisy Data and Higher Dimensional Problems}
\label{sect:noisy}
The above algorithm was derived based on the fact that the \SNL had exact data;
i.e., for a given clique $\alpha$ we had an exact correspondence between the 
\edm and the corresponding Gram matrix $B=\KK^\dagger (D[\alpha])$.
To extend this to the noisy case, we apply a naive, greedy approach. When
the Gram matrix $B$ is needed, then we use the best rank $r$ positive
semidefinite approximation to $B$ using the well-known Eckert-Young
result; see e.g., \cite[Cor.~2.3.3]{GoVan:79}.
\begin{lem}
Suppose that $B \in \Sn$ with spectral decomposition $B=\sum_{i=1}^n
\lambda_i u_iu_i^T, \lambda_1 \geq \ldots \geq \lambda_n$. 
Then the best positive semidefinite approximation with at
most rank $r$ is $B_+= \sum_{i=1}^r (\lambda_i)_+ u_iu_i^T$, where 
$(\lambda_i)_+ = \max \{0, \lambda_i\}$.
\qed
\end{lem}

We follow the  multiplicative noise model in, e.g.,~\cite{BYIEEE:06,BY:04,KimKojimaWaki:09,pongtseng:09,tseng:07,WangZhengBoydYe:06}; 
i.e.,  the noisy (squared) distances $D_{ij}$ are given by
\[
D_{ij} = \left( \|p_i-p_j\|(1+\sigma\epsilon_{ij}) \right)^2,
\]
where $\sigma \geq 0$ is the noise factor and $\epsilon_{ij}$ is chosen
from the standard normal distribution $\N(0,1)$.  
We include preliminary test results in Table~\ref{table:NoisyDim23} for problems with 0\%-1\% noise 
with embedding dimension $r=2,3$.  Note that we do not apply the noise to the distances between the anchors.

\begin{table}
\label{table:NoisyDim23}
\caption[Problems with noise and $r = 2,3$]{Results of Algorithm~\ref{alg:SNLSDPclique} 
for problems with noise and $r = 2,3$, using {\texttt{RigidCliqueUnion}} and {\texttt{RigidNodeAbsorb}}.  The values for Average Degree, \# Sensors Positioned, CPU Time, Max~Error and RMSD are averaged over ten random instances.}
\begin{center}
\begin{footnotesize}
\begin{tabular}{|c|c|c|c|c||c|c|c|c|c|}
\hline
         &            &     &            &     & Average & \# Sensors &          &           &      \\           
$\sigma$ & \# sensors & $r$ & \# anchors & $R$ &  Degree & Positioned & CPU Time & Max Error & RMSD \\
\hline
\hline
0    &  2000 & 2 & 4 & .08 & 18.8 &  2000.0 & 1 s & 1e-13 & 3e-14 \\
1e-6 &  2000 & 2 & 4 & .08 & 18.8 &  2000.0 & 1 s & 2e-04 & 4e-05 \\
1e-4 &  2000 & 2 & 4 & .08 & 18.8 &  2000.0 & 1 s & 2e-02 & 4e-03 \\
1e-2 &  2000 & 2 & 4 & .08 & 18.8 &  2000.0 & 1 s & 2e+01 & 3e+00 \\
\hline
0    &  6000 & 2 & 4 & .06 & 32.3 &  6000.0 & 4 s & 2e-13 & 7e-14 \\
1e-6 &  6000 & 2 & 4 & .06 & 32.3 &  6000.0 & 4 s & 8e-04 & 3e-04 \\
1e-4 &  6000 & 2 & 4 & .06 & 32.3 &  6000.0 & 4 s & 9e-02 & 3e-02 \\
1e-2 &  6000 & 2 & 4 & .06 & 32.3 &  6000.0 & 4 s & 2e+01 & 3e+00 \\
\hline
0    & 10000 & 2 & 4 & .04 & 24.3 & 10000.0 & 6 s & 3e-13 & 1e-13 \\
1e-6 & 10000 & 2 & 4 & .04 & 24.3 & 10000.0 & 6 s & 5e-04 & 2e-04 \\
1e-4 & 10000 & 2 & 4 & .04 & 24.3 & 10000.0 & 6 s & 5e-02 & 2e-02 \\
1e-2 & 10000 & 2 & 4 & .04 & 24.3 & 10000.0 & 7 s & 4e+02 & 1e+02 \\
\hline
\hline
0    &  2000 & 3 & 5 & .20 & 26.6 &  2000.0 & 1 s & 3e-13 & 8e-14 \\
1e-6 &  2000 & 3 & 5 & .20 & 26.6 &  2000.0 & 1 s & 7e-04 & 2e-04 \\
1e-4 &  2000 & 3 & 5 & .20 & 26.6 &  2000.0 & 1 s & 8e-02 & 2e-02 \\
1e-2 &  2000 & 3 & 5 & .20 & 26.6 &  2000.0 & 1 s & 2e+03 & 4e+02 \\
\hline
0    &  6000 & 3 & 5 & .15 & 35.6 &  6000.0 & 5 s & 3e-13 & 6e-14 \\
1e-6 &  6000 & 3 & 5 & .15 & 35.6 &  6000.0 & 5 s & 1e-03 & 2e-04 \\
1e-4 &  6000 & 3 & 5 & .15 & 35.6 &  6000.0 & 5 s & 1e-01 & 2e-02 \\
1e-2 &  6000 & 3 & 5 & .15 & 35.6 &  6000.0 & 6 s & 9e+01 & 9e+00 \\
\hline
0    & 10000 & 3 & 5 & .10 & 18.7 & 10000.0 &  9 s & 3e-12 & 2e-13 \\
1e-6 & 10000 & 3 & 5 & .10 & 18.7 & 10000.0 & 10 s & 4e-02 & 2e-03 \\
1e-4 & 10000 & 3 & 5 & .10 & 18.7 & 10000.0 & 10 s & 2e+00 & 8e-02 \\
1e-2 & 10000 & 3 & 5 & .10 & 18.7 & 10000.0 & 10 s & 4e+02 & 1e+01 \\
\hline
\end{tabular}
\end{footnotesize}
\end{center}
\end{table}

\section{Conclusion}
\label{sect:concl}
The \SDP relaxation of \SNL is highly (implicitly) degenerate, since
the feasible set of this \SDP is restricted to a low dimensional
face of the \SDP cone, resulting in the failure of the Slater constraint 
qualification (strict feasibility).
We take advantage of this degeneracy by finding
explicit representations of intersections
of faces of the \SDP cone corresponding to unions of intersecting cliques. 
In addition, from these representations we force further degeneracy in
order to find the minimal face that contains the optimal solution.
In many cases, we can efficiently compute the exact solution to the \SDP relaxation
\emph{without} using any \SDP solver.

In some cases it is not possible to reduce the problem down to a single clique.  However, in these cases, the intersection of the remaining faces returned by {\texttt{SNLSDPclique}} will produce a face containing the feasible region of the original problem.  This face can then be used to reduce the problem before passing the problem to an \SDP solver, where, for example, the trace of the semidefinite matrix can be maximized \cite{MR2398864} to try to keep the embedding dimension small.   As an example, if the problem is composed of disjoint cliques, then Corollary~\ref{cor:disjcliques} can be used to significantly reduce the problem size.  This reduction can transform a large intractable problem into a much smaller problem that can be solved efficiently via an \SDP solver.

\clearpage
\phantomsection
\addcontentsline{toc}{section}{References}
\bibliography{.master,.edm,.psd,.bjorBOOK}

\def\cprime{$'$} \def\cprime{$'$} \def\cprime{$'$}
  \def\udot#1{\ifmmode\oalign{$#1$\crcr\hidewidth.\hidewidth
  }\else\oalign{#1\crcr\hidewidth.\hidewidth}\fi} \def\cprime{$'$}
  \def\cprime{$'$} \def\cprime{$'$}
\begin{thebibliography}{10}

\bibitem{homwolkA:04}
S.~AL-HOMIDAN and H.~WOLKOWICZ.
\newblock Approximate and exact completion problems for {E}uclidean distance
  matrices using semidefinite programming.
\newblock {\em Linear Algebra Appl.}, 406:109--141, 2005.

\bibitem{AlfakihAnjosKPW:08}
A.~ALFAKIH, M.F. ANJOS, V.~PICCIALLI, and H.~WOLKOWICZ.
\newblock Euclidean distance matrices, semidefinite programming, and sensor
  network localization.
\newblock Technical Report CORR 2009-05, University of Waterloo, Waterloo,
  Ontario, 2009.

\bibitem{AlKaWo:97}
A.~ALFAKIH, A.~KHANDANI, and H.~WOLKOWICZ.
\newblock Solving {E}uclidean distance matrix completion problems via
  semidefinite programming.
\newblock {\em Comput. Optim. Appl.}, 12(1-3):13--30, 1999.
\newblock Computational optimization---a tribute to Olvi Mangasarian, Part I.

\bibitem{AmVav:09}
B.~AMES and S.A. VAVASIS.
\newblock Nuclear norm minimization for the planted clique and biclique
  problems.
\newblock Technical report, University of Waterloo, 2009.

\bibitem{biswasphd07}
P.~BISWAS.
\newblock {\em Semidefinite programming approaches to distance geometry
  problems}.
\newblock PhD thesis, Stanford University, 2007.

\bibitem{BYIEEE:06}
P.~BISWAS, T.-C. LIANG, Y.~YE, K-C. TOH, and T.-C. WANG.
\newblock Semidefinite programming approaches for sensor network localization
  with noisy distance measurements.
\newblock {\em IEEE Transactions onAutomation Science and Engineering},
  3(4):360--371, October 2006.

\bibitem{biswasliangtohwangye}
P.~BISWAS, T.C. LIANG, K.C. TOH, T.C. WANG, and Y.~YE.
\newblock Semidefinite programming approaches for sensor network localization
  with noisy distance measurements.
\newblock {\em IEEE Transactions on Automation Science and Engineering},
  3:360--371, 2006.

\bibitem{biswasliangtohye:05}
P.~BISWAS, T.C. LIANG, K.C. TOH, and Y.~YE.
\newblock An {SDP} based approach for anchor-free {3D} graph realization.
\newblock Technical report, Operation Research, Stanford University, Stanford,
  CA, 2005.

\bibitem{MR2398864}
P.~BISWAS, K.C. TOH, and Y.~YE.
\newblock A distributed {SDP} approach for large-scale noisy anchor-free graph
  reailzation with applications to molecular conformation.
\newblock {\em SIAM J. Sci. Comput.}, 30(3):1251--1277, 2008.

\bibitem{BiswasYe:04}
P.~BISWAS and Y.~YE.
\newblock Semidefinite programming for ad hoc wireless sensor network
  localization.
\newblock In {\em Information Processing In Sensor Networks, Proceedings of the
  third international symposium on Information processing in sensor networks},
  pages 46--54, Berkeley, Calif., 2004.

\bibitem{BY:04}
P.~BISWAS and Y.~YE.
\newblock Semidefinite programming for ad hoc wireless sensor network
  localization.
\newblock In {\em IPSN '04: Proceedings of the 3rd international symposium on
  Information processing in sensor networks}, pages 46--54, New York, NY, USA,
  2004. ACM.

\bibitem{MR2191577}
P.~BISWAS and Y.~YE.
\newblock A distributed method for solving semidefinite programs arising from
  ad hoc wireless sensor network localization.
\newblock In {\em Multiscale optimization methods and applications}, volume~82
  of {\em Nonconvex Optim. Appl.}, pages 69--84. Springer, New York, 2006.

\bibitem{MR2274505}
M.W. CARTER, H.H. JIN, M.A. SAUNDERS, and Y.~YE.
\newblock Spase{L}oc: an adaptive subproblem algorithm for scalable wireless
  sensor network localization.
\newblock {\em SIAM J. Optim.}, 17(4):1102--1128, 2006.

\bibitem{cassioli:08}
A.~CASSIOLI.
\newblock {\em Global optimization of highly multimodal problems}.
\newblock PhD thesis, Universita di Firenze, Dipartimento di sistemi e
  informatica, Via di S.Marta 3, 50139 Firenze, Italy, 2008.

\bibitem{dattorro:05}
J.~DATTORRO.
\newblock {\em Convex Optimization \& {E}uclidean Distance Geometry}.
\newblock Meboo Publishing, USA, 2005.

\bibitem{DiKrQiWo:06}
Y.~DING, N.~KRISLOCK, J.~QIAN, and H.~WOLKOWICZ.
\newblock Sensor network localization, {E}uclidean distance matrix completions,
  and graph realization.
\newblock {\em Optimization and Engineering}, to appear(CORR 2006-23, to
  appear), 2006.

\bibitem{DiKrQiWo:08}
Y.~DING, N.~KRISLOCK, J.~QIAN, and H.~WOLKOWICZ.
\newblock Sensor network localization, {E}uclidean distance matrix completions,
  and graph realization.
\newblock In {\em Proceedings of MELT08, San Francisco}, pages 129--134, 2008.

\bibitem{egwymab04}
T.~EREN, D.K. GOLDENBERG, W.~WHITELEY, Y.R. YANG, A.S. MORSE, B.D.O. ANDERSON,
  and P.N. BELHUMEUR.
\newblock Rigidity, computation, and randomization in network localization,
  2004.
\newblock IEEE INFOCOM.

\bibitem{FeiKrau:00}
U.~FEIGE and R~KRAUTHGAMER.
\newblock Finding and certifying a large hidden clique in a semi-random graph.
\newblock {\em Random Structures and Algorithms}, 16(2):195--202, 2000.

\bibitem{GoVan:79}
G.H. GOLUB and C.F. {VAN LOAN}.
\newblock {\em Matrix Computations}.
\newblock Johns Hopkins University Press, Baltimore, Maryland, $3^{nd}$
  edition, 1996.

\bibitem{hendrickson:91}
B.~HENDRICKSON.
\newblock {\em The molecule problem: Determining conformation from pairwise
  distances}.
\newblock PhD thesis, Cornell University, Ithaca, New York, 1991.

\bibitem{MR92m:05182}
B.~HENDRICKSON.
\newblock Conditions for unique graph realizations.
\newblock {\em SIAM J. Comput.}, 21(1):65--84, 1992.

\bibitem{Jin:05}
H.~JIN.
\newblock {\em Scalable Sensor Localization Algorithms for Wireless Sensor
  Networks}.
\newblock PhD thesis, Toronto University, Toronto, Ontario, Canada, 2005.

\bibitem{KimKojimaWaki:09}
S.~KIM, M.~KOJIMA, and H.~WAKI.
\newblock Exploiting sparsity in {SDP} relaxation for sensor network
  localization.
\newblock {\em SIAM J. Optim.}, 20(1):192--215, 2009.

\bibitem{PatakiSVW:99}
G.~PATAKI.
\newblock Geometry of {S}emidefinite {P}rogramming.
\newblock In H.~Wolkowicz, R.~Saigal, and L.~Vandenberghe, editors, {\em
  HANDBOOK OF SEMIDEFINITE PROGRAMMING: Theory, Algorithms, and Applications}.
  Kluwer Academic Publishers, Boston, MA, 2000.

\bibitem{pongtseng:09}
T.K. PONG and P.~TSENG.
\newblock ({R}obust) edge-based semidefinite programming relaxation of sensor
  network localization.
\newblock Technical Report Jan-09, University of Washington, Seattle, WA, 2009.

\bibitem{sax79}
J.~B. SAXE.
\newblock Embeddability of weighted graphs in k-space is strongly {NP}-hard.
\newblock {\em Proc. 17th {A}llerton Conf. in Communications, Control, and
  Computing}, pages 480--489, 1979.

\bibitem{MR2295148}
A.M-C. SO and Y.~YE.
\newblock Theory of semidefinite programming for sensor network localization.
\newblock {\em Math. Program.}, 109(2-3, Ser. B):367--384, 2007.

\bibitem{tseng:07}
P.~TSENG.
\newblock Second-order cone programming relaxation of sensor network
  localization.
\newblock {\em SIAM J. on Optimization}, 18(1):156--185, 2007.

\bibitem{MR97b:90077}
R.J. VANDERBEI and Y.~BING.
\newblock The simplest semidefinite programs are trivial.
\newblock {\em Math. Oper. Res.}, 20(3):590--596, 1995.

\bibitem{WangZhengBoydYe:06}
Z.~WANG, S.~ZHENG, S.~BOYD, and Y.~YE.
\newblock Further relaxations of the semidefinite programming approach to
  sensor network localization.
\newblock {\em SIAM J. Optim.}, 19(2):655--673, 2008.

\bibitem{Wolk:93}
H.~WOLKOWICZ.
\newblock Explicit solutions for interval semidefinite linear programs.
\newblock {\em Linear Algebra Appl.}, 236:95--104, 1996.

\end{thebibliography}

\clearpage
\phantomsection
\addcontentsline{toc}{section}{Index}
\printindex

\end{document}